\documentclass[11pt, a4paper]{amsart}
\usepackage{graphicx}
\usepackage[all]{xy}
\usepackage{hyperref}

\newcommand{\ko}{\;\; ,}

\numberwithin{equation}{subsection}
\newtheorem{theorem}[subsection]{Theorem}
\newtheorem{classification-theorem}[subsection]{Classification Theorem}
\newtheorem{decomposition-theorem}[subsection]{Decomposition Theorem}
\newtheorem{proposition-definition}[subsection]{Proposition-Definition}
\newtheorem{periodicity-conjecture}[subsection]{Periodicity Conjecture}
\newtheorem{lemma}[subsection]{Lemma}
\newtheorem{proposition}[subsection]{Proposition}
\newtheorem{corollary}[subsection]{Corollary}

\newtheorem{question}[subsection]{Question}
\newtheorem{assumption}[subsection]{Assumption}

\newtheorem{remark}[subsection]{Remark}

\newcommand{\rad}{\operatorname{rad}\nolimits}
\newcommand{\res}{\operatorname{res}\nolimits}
\newcommand{\can}{\operatorname{can}\nolimits}

\newcommand{\Gr}{\operatorname{Gr}\nolimits}
\newcommand{\RHom}{\operatorname{RHom}\nolimits}
\newcommand{\Rlim}{\operatorname{Rlim}\nolimits}

\newcommand{\Hom}{\operatorname{Hom}\nolimits}
\newcommand{\im}{\operatorname{im}\nolimits}
\newcommand{\cok}{\operatorname{cok}\nolimits}

\newcommand{\Ext}{\operatorname{Ext}\nolimits}
\newcommand{\Ex}{\operatorname{E}\nolimits}

\newcommand{\Mod}{\operatorname{Mod }\nolimits}
\renewcommand{\mod}{\operatorname{mod}\nolimits}

\newcommand{\rep}{\operatorname{rep}\nolimits}
\newcommand{\proj}{\operatorname{proj}\nolimits}
\newcommand{\inj}{\operatorname{inj}\nolimits}
\newcommand{\Lex}{\operatorname{Lex}\nolimits}
\newcommand{\ind}{\operatorname{ind}\nolimits}

\newcommand{\Z}{\operatorname{\mathbb{Z}}\nolimits}

\newcommand{\N}{\operatorname{\mathbb{N}}\nolimits}

\newcommand{\gpr}{\operatorname{gpr}\nolimits}
\newcommand{\gin}{\operatorname{gin}\nolimits}

\newcommand{\GL}{\operatorname{GL}\nolimits}

\newcommand{\ten}{\otimes}
\newcommand{\lten}{\overset{\mathbf{L}}{\ten}}
\newcommand{\dimv}{\ul{\dim}\,}

\newcommand{\bt}{\bullet}

\newcommand{\iso}{\stackrel{_\sim}{\rightarrow}}
\newcommand{\liso}{\stackrel{_\sim}{\leftarrow}}
\newcommand{\id}{\mathbf{1}}
\newcommand{\arr}[1]{\stackrel{#1}{\rightarrow}}

\newcommand{\ca}{{\mathcal A}}

\newcommand{\cc}{{\mathcal C}}
\newcommand{\cd}{{\mathcal D}}
\newcommand{\ce}{{\mathcal E}}

\newcommand{\cm}{{\mathcal M}}
\newcommand{\cn}{{\mathcal N}}

\newcommand{\cR}{{\mathcal R}}

\newcommand{\cs}{{\mathcal S}}
\newcommand{\ct}{{\mathcal T}}

\newcommand{\cv}{{\mathcal V}}

\newcommand{\fg}{\mathfrak{g}}

\newcommand{\eps}{\varepsilon}

\renewcommand{\tilde}[1]{\widetilde{#1}}

\newcommand{\ul}[1]{\underline{#1}}
\newcommand{\ol}[1]{\overline{#1}}
\renewcommand{\hat}[1]{\widehat{#1}}

\renewcommand{\MR}[1]{}

\begin{document}
\title[Graded quiver varieties and derived dategories]{Graded quiver varieties\\ and derived categories}
\author{Bernhard Keller and Sarah Scherotzke}
\address{B.~K.~: Universit\'e Paris Diderot -- Paris~7, Institut
Universitaire de France, UFR de Math\'ematiques, Institut de
Math\'ematiques de Jussieu, UMR 7586 du CNRS, Case 7012, B\^atiment
Sophie Germain, 75205 Paris Cedex 13, France}
\address{S.~S.~: University of Bonn, Mathematisches Institut, 
Endenicher Allee 60, 53115 Bonn, Germany}

\email{keller@math.jussieu.fr, sarah@math.uni-bonn.de}


\dedicatory{To the memory of Dieter Happel}

\begin{abstract} Inspired by recent work of Hernandez--Leclerc and
Leclerc--Pla\-mon\-don we investigate the link between Nakajima's graded
affine quiver varieties associated with an acyclic connected quiver $Q$ and the
derived category of $Q$. As Leclerc--Plamondon have shown,
the points of these varieties can be interpreted
as representations of a category, which we call the (singular) Nakajima
category $\mathcal{S}$. We determine the quiver of $\mathcal{S}$ 
and the number of minimal relations between any two given vertices. 
We construct a $\delta$-functor $\Phi$ taking each finite-dimensional 
representation of $\mathcal{S}$ to an object of the derived category of $Q$. 
We show that the functor $\Phi$ establishes a bijection between the strata of the 
graded affine quiver varieties and the isomorphism classes of objects 
in the image of $\Phi$. If the underlying graph of $Q$ is an ADE Dynkin
diagram, the image is the whole derived category; otherwise, it is
the category of `line bundles over the non commutative curve given by $Q$'.
We show that the degeneration order between
strata corresponds to Jensen--Su--Zimmermann's degeneration
order on objects of the derived category. Moreover, if $Q$ is an ADE Dynkin quiver,
the singular category $\cs$ is weakly Gorenstein of dimension $1$ and its
derived category of singularities is equivalent to the derived category of $Q$.
\end{abstract}

\maketitle

\tableofcontents

\section{Introduction}

Let $Q$ be a Dynkin quiver, i.e. a quiver whose underlying graph is an ADE Dynkin
diagram $\Delta$. The (affine) graded quiver varieties associated with 
$Q$ were introduced by Nakajima in \cite{Nakajima01}. 
In type $A$, they generalize Ginzburg--Vasserot's graded nilpotent orbit closures 
\cite{GinzburgVasserot93}.
They have been of great importance in 
\begin{itemize}
\item[1)] Nakajima's geometric study \cite{Nakajima01}
of the finite-dimensional representations of the quantum affine algebra 
$U_q(\hat{\fg})$ associated with $\Delta$,
\item[2)] his related study of cluster algebras in \cite{Nakajima11} \cite{Nakajima13},
cf. also the survey \cite{Leclerc11}.
\end{itemize}
Let us elaborate on the second point: In \cite{Nakajima11}, 
Nakajima showed how to use
categories of perverse sheaves on graded quiver
varieties in order to investigate the cluster algebra $\ca_Q$
associated with $Q$ by Fomin--Zelevinsky \cite{FominZelevinsky02}. 
He did so not only for Dynkin quivers but more generally for arbitrary
bipartite quivers (where each vertex either has only incoming or
only outgoing arrows). He showed that the dual Grothendieck ring associated
with these categories (almost) yields a monoidal categorification of $\ca_Q$
in the sense of Hernandez--Leclerc \cite{HernandezLeclerc10},
who had constructed monoidal categorifications in types
$A_n$ and $D_4$ (they extend their results to all linearly
oriented quivers of type $A$ or $D$ in their recent article
\cite{HernandezLeclerc12}). Qin \cite{Qin12} \cite{Qin12a}
has generalized Nakajima's construction of graded quiver
varieties to all acyclic quivers $Q$ and Kimura--Qin 
\cite{KimuraQin12} have used these varieties to
extend Nakajima's results on cluster algebras to
this generality.

In section~9 of their remarkable study \cite{HernandezLeclerc11}
of deformed Grothendieck rings of quantum affine
algebras, Hernandez--Leclerc proved that the
graded quiver varieties associated with certain special weights 
are isomorphic to varieties of representations of $Q$ in such a way
that Nakajima's stratification corresponds to the
natural stratification by orbits. This description was
extended by Leclerc--Plamondon \cite{LeclercPlamondon12}, 
who showed that the quiver varieties in a much larger class are isomorphic
to varieties of representations of the repetitive 
algebra \cite{HughesWaschbuesch83} \cite{Happel87}
associated with $Q$, where Nakajima's stratification
again corresponds to the natural one by orbits. Let
us call LP-varieties the graded quiver varieties covered by
Leclerc-Plamondon's construction. 
Via Happel's equivalence \cite{Happel87} between the
stable category of the repetitive algebra of $Q$ and the derived
category of $Q$, Leclerc--Plamondon's isomorphism yields a
map from a given LP-variety to the set of isomorphism
classes of the derived category of $Q$ and,
as shown in \cite{LeclercPlamondon12}, the fibers of this map
are precisely the Nakajima strata. In this article, we
extend this last result in two directions simultaneously:
\begin{itemize}
\item[1)] from LP-varieties to all graded quiver varieties,
\item[2)] from Dynkin quivers to arbitrary acyclic quivers (using
Qin's definition \cite{Qin12} \cite{Qin12a} of graded quiver varieties).
\end{itemize}
Along the way, we obtain information on graded affine quiver
varieties as well as on their desingularization by Nakajima's
smooth (quasi-projective) graded quiver varieties. Among other results, 
\begin{itemize}
\item[-] we determine the quiver of the singular Nakajima category
 $\cs$, whose representations form the (affine) graded quiver varieties;
\item[-] we determine the number of minimal relations between 
the vertices of the quiver of $\cs$; remarkably, there are {\em no} relations
if $Q$ is a connected non Dynkin quiver;
\item[-] we construct the stratifying functor $\Phi$ from the category
of finite-dimensional $\cs$-modules to the derived category of $Q$
and use it to describe the strata and their closures in terms of the
derived category;
 \item[-] we describe the fibers of Nakajima's desingularization map using $\Phi$ 
 in the spirit of theorems by Lusztig \cite{Lusztig98a}, 
Savage--Tingley \cite{SavageTingley11} and Shipman \cite{Shipman10};
\item[-] we extend Happel's equivalence  \cite{Happel87} by showing that,
for a Dynkin quiver $Q$, the singular category $\cs$  is weakly Gorenstein and that
its derived category of singularities is equivalent to the derived category of $Q$;
\item[-] we vastly generalize the preceding point by showing that 
for any configuration $C$ of vertices of $\cs$ satisfying
a certain natural condition, the associated quotient $\cs_C$ of
$\cs$ is weakly Gorenstein with associated derived category of singularities
equivalent to the derived category of $Q$.
\end{itemize}
We refer to section~\ref{s:notation-and-main-results} for a more
detailed description of our main results. In the companion paper
\cite{KellerScherotzke13b}, we show how to use
Nakajima's desingularization map to generalize recent results by 
Cerulli--Feigin--Reineke \cite{CerulliFeiginReineke12} 
\cite{CerulliFeiginReineke13} on quiver Grassmannians.

Let us emphasize that throughout, we use framed quiver varieties.
As shown by Crawley--Boevey \cite{CrawleyBoevey01}, from the
point of view of the geometry of the individual quiver varieties, the framing
may be neglected. However, it is essential in the applications
to quantum affine algebras and cluster algebras
alluded to above as well as in the homological approach we use.
We hope to come back to the relation of this approach with
that of Frenkel--Khovanov--Schiffmann \cite{FrenkelKhovanovSchiffmann05}
in future work.

\subsection*{Acknowledgments} 
A large part of the work on this article was done during the 
cluster algebra program at the MSRI in fall 2012. The authors are 
grateful to the MSRI for financial support and ideal working conditions.
They are indebted to Bernard Leclerc and Pierre-Guy Plamondon for informing
them about the main results of \cite{LeclercPlamondon12} prior to
its appearance on the archive. They are obliged to Osamu Iyama for
pointing out reference \cite{Iyama05b} and to Harold Williams for
asking a question that lead to Theorem~\ref{thm:degeneration-order}.
They thank Giovanni Cerulli Irelli, David Hernandez, Osamu Iyama, Bernard Leclerc,
Pierre-Guy Plamondon, Fan Qin and Markus Reineke for stimulating conversations
and for helpful comments on a preliminary version of this article.

\section{Notation and main results}
\label{s:notation-and-main-results}

\subsection{Repetition quivers and Happel's theorem}
\label{ss:repetition-quivers-and-Happels-theorem}
Let $Q$ be a quiver. Thus, $Q$ is an oriented graph given by a set of vertices 
$Q_0$, a set of arrows $Q_1$ and two maps $s:Q_1\to Q_0$ and 
$t:Q_1 \to Q_0$ taking an arrow to its source vertex respectively its 
target vertex. We assume that $Q$ is finite (both $Q_0$ and $Q_1$ are finite) 
and acyclic (there are no oriented cycles in $Q$). 

The {\em repetition quiver $\Z Q$}, cf. \cite{Riedtmann80a}, has the set of 
vertices $\Z Q_0$ formed by all pairs $(i,p)$, where $i$ belongs to $Q_0$ 
and $p$ is an integer. For
each arrow $\alpha: i \to j$, it has the arrows $(\alpha, p): (i,p) \to (j,p)$
and $\sigma(\alpha,p): (j,p-1) \to (i,p)$, where $p$ runs through the integers.
If $\beta$ is an arbitrary arrow of $\Z Q$, we put $\sigma(\beta)=\sigma(\alpha,p)$
if $\beta=(\alpha,p)$ and $\sigma(\beta)=(\alpha, p-1)$ if $\beta=\sigma(\alpha,p)$.
We denote by $\tau: \Z Q \to \Z Q$ the automorphism of $\Z Q$ given by
the left translation by one unit: we have $\tau(i,p)=(i,p-1)$ and
$\tau(\beta)=\sigma^2(\beta)$ for all $i\in Q_0$, $p\in \Z$, and for all
arrows $\beta$ of $\Z Q$. For example, when $Q$ is the quiver $1 \to 2 \to 3$, the
repetition quiver has the form given in Figure~\ref{fig:repetition-quiver-A3}.
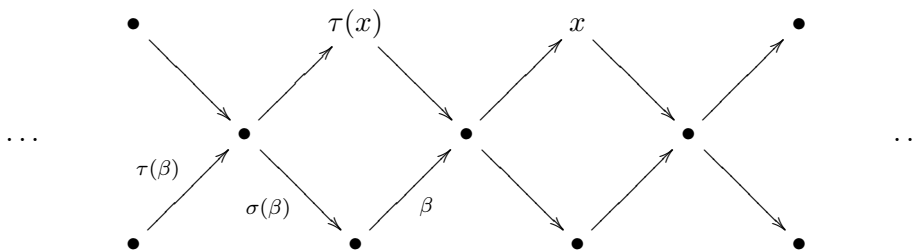
\begin{figure}
\[
\xymatrix@=0.5cm@!{
 & \bt \ar[dr] & & \tau(x) \ar[dr] & & x \ar[dr] & & \bt &  \\
 \ldots & & \bt \ar[ur] \ar[dr]_{\sigma(\beta)}  & & \bt \ar[ur] \ar[dr] & & \bt \ar[ur] \ar[dr] & & \ldots \\
  & \bt \ar[ur]^{\tau(\beta)} & & \bt \ar[ur]_\beta & & \bt \ar[ur] & & \bt & 
  }
\]
\caption{The repetition quiver $\Z Q$ for $Q$ of type $A_3$}
\label{fig:repetition-quiver-A3}
\end{figure}

Let $k$ be a field. Following \cite{Gabriel80} \cite{Riedtmann80}, we
define the {\em mesh category $k(\Z Q)$} to be the $k$-category whose
objects are the vertices of $\Z Q$ and whose morphism space from
$a$ to $b$ is the space of all $k$-linear combinations of paths from
$a$ to $b$ modulo the subspace spanned by all elements $u r_x v$, where
$u$ and $v$ are paths and 
\[
r_x = \sum_{\beta: y \to x} \sigma(\beta) \beta: \quad
\raisebox{1.225cm}{\xymatrix@R=0.5cm@C=0.5cm{  & y_1 \ar[dr]^{\beta_1} & \\
\tau(x) \ar[ur]^{\sigma(\beta_1)} \ar[dr]_{\sigma(\beta_s)}  & \vdots & x \\
 & y_s \ar[ur]_{\beta_s} & }}
\]
is the {\em mesh relator} associated with a vertex $x$ of $\Z Q$. Here the
sum runs over all arrows $\beta: y \to x$ of $\Z Q$. For
example, in the mesh category $k(\Z \vec{A}_2)$ associated with
the quiver $Q=\vec{A}_2: 1 \to 2$, the composition of any two
consecutive arrows vanishes. The computation of the morphism spaces in 
$k(\Z Q)$ is easy using additive functions, cf. section~6.5 of \cite{Gabriel80}.

Let $kQ$ be the path algebra of $Q$. It is a finite-dimensional, hereditary
$k$-algebra. For each vertex $i$ of $Q$, we write $e_i$ for the associated idempotent
of $kQ$ (the `lazy path at $i$') and $P_i = e_i kQ$ for the indecomposable
projective $kQ$-module whose head is the simple module $S_i$ concentrated
at the vertex $i$. Let
$\mod kQ$ be the category of all $k$-finite-dimensional right $kQ$-modules.
Let $\cd_Q$ be the bounded derived category $\cd^b(\mod kQ)$. 
It is a Krull--Schmidt category \cite{Happel87} and a triangulated
category. We write $\Sigma$ for its shift (=suspension) functor. Let
$\ind(\cd_Q)$ be a full subcategory of $\cd_Q$ whose objects form
a set of representatives of the isomorphism classes of indecomposable
objects of $\cd_Q$. The following theorem is Proposition~4.6 of \cite{Happel87}
and Theorem~5.6 of \cite{Happel88}.

\begin{theorem}[Happel, 1987] 
\label{thm:Happel}
There is a canonical fully faithful functor
\[
H: k(\Z Q) \to \ind(\cd_Q)
\]
taking each vertex $(i,0)$ to the indecomposable
projective module $P_i$, $i\in Q_0$. It is an equivalence iff
$Q$ is a Dynkin quiver (i.e. its underlying graph is a disjoint
union of ADE Dynkin diagrams).
\end{theorem}

Let $\nu: \cd_Q \to \cd_Q$ be the autoequivalence given by the derived tensor product
with the $k$-dual of $kQ$ considered as a bimodule. 
We have an isomorphism, bifunctorial in $L,M \in\cd_Q$,
\begin{equation} \label{eq:Serre-duality}
D\Hom_{\cd_Q}(L,M) = \Hom_{\cd_Q}(M, \nu L) \ko
\end{equation}
where $D$ denotes the duality over $k$. This
means that $\nu$ is the  {\em Serre functor of $\cd_Q$}. 
As shown in \cite{Happel87}, via the embedding $H$,
the autoequivalence $\tau$ of the mesh category corresponds to
the {\em Auslander--Reiten translation $\tau_{\cd_Q}=\Sigma^{-1} \nu$}, 
which we will also denote by $\tau$. For Dynkin quivers,
the combinatorial descriptions of $\nu$ (equivalently: $\Sigma$) and of the image of 
$\mod kQ$ in $\cd_Q$ are given in section~6.5 of
\cite{Gabriel80}.

For later use, we record the following isomorphism, which
follows from Serre duality (\ref{eq:Serre-duality}):  For  $L, M\in\cd_Q$
and $p\in\Z$, we have
\begin{equation} \label{eq:Auslander-Reiten-formula}
D\Ext^p_{\cd_Q}(L, M) = \Hom_{\cd_Q}(M, \Sigma^{-(p-1)} \tau L) \ko
\end{equation}
where, as usual, we write $\Ext^p_{\cd_Q}(L,M)$ for $\Hom_{\cd_Q}(L, \Sigma^p M)$.

\subsection{Graded affine quiver varieties} \label{ss:graded-quiver-varieties}
Let $Q$ be a finite acyclic quiver as in section~\ref{ss:repetition-quivers-and-Happels-theorem} and let $k$ be the field of complex numbers. The {\em framed quiver $\tilde{Q}$} is obtained from $Q$ by adding, for each vertex $i$, a new vertex $i'$ and
a new arrow $i \to i'$. For example, if $Q$ is the quiver $1 \to 2$, the
framed quiver is
\[ 
\xymatrix{ 
2 \ar[r] & 2' \\
1 \ar[r] \ar[u] & 1' .
} 
\]
Let $\Z\tilde{Q}$ be the repetition quiver of $\tilde{Q}$. We refer to the
vertices $(i', p)$, $i\in Q_0$, $p\in \Z$, as the {\em frozen vertices} of $\Z \tilde{Q}$
and mark them by squares as in the examples in
Figure~\ref{fig:repetition-quivers-A2-D4} associated with
quivers whose underlying graphs are $A_2$ respectively $D_4$.
For a vertex $x=(i,p)$, we put $\sigma(x)=(i', p-1)$
and for a vertex $(i', p)$, we put $\sigma(i',p)=(i, p-1)$.

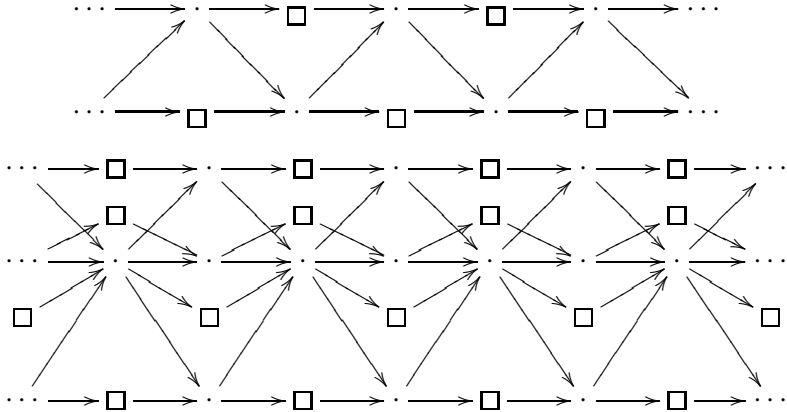
\begin{figure}
\[ 
\xymatrix{ \cdots  \ar[r]  & \cdot \ar[r] \ar[rd] & \boxed{ } \ar[r]  &  \cdot \ar[r]   \ar[rd] &  \boxed{ } \ar[r] & \cdot \ar[r]   \ar[rd] & \cdots \\
\cdots  \ar[r]  \ar[ru]  & \boxed{}  \ar[r]  &  \cdot  \ar[r]  \ar[ru]  &  \boxed{ } \ar[r]  & \cdot  \ar[r]  \ar[ru]  & \boxed{} \ar[r]  &  \cdots }
\]
\[
\begin{xy} 0;<0.7pt,0pt>:<0pt,-0.7pt>:: 
(50,80) *+{\boxed{}} ="1",
(50,0) *+{\cdots} ="2",
(50,50) *+{\cdots} ="3",
(50,125) *+{\cdots} ="4",
(100,0) *+{\boxed{}} ="5",
(100,25) *+{\boxed{}} ="6",
(100,125) *+{\boxed{}} ="7",
(100,50) *+{\cdot} ="8",
(150,80) *+{\boxed{}} ="9",
(150,0) *+{\cdot} ="10",
(150,50) *+{\cdot} ="11",
(150,125) *+{\cdot} ="12",
(200,0) *+{\boxed{}} ="13",
(200,25) *+{\boxed{}} ="14",
(200,125) *+{\boxed{}} ="15",
(200,50) *+{\cdot} ="16",
(250,80) *+{\boxed{}} ="17",
(250,0) *+{\cdot} ="18",
(250,50) *+{\cdot} ="19",
(250,125) *+{\cdot} ="20",
(300,0) *+{\boxed{}} ="21",
(300,25) *+{\boxed{}} ="22",
(300,125) *+{\boxed{}} ="23",
(300,50) *+{\cdot} ="24",
(350,80) *+{\boxed{}} ="25",
(350,0) *+{\cdot} ="26",
(350,50) *+{\cdot} ="27",
(350,125) *+{\cdot} ="28",
(400,0) *+{\boxed{}} ="29",
(400,25) *+{\boxed{}} ="30",
(400,125) *+{\boxed{}} ="31",
(400,50) *+{\cdot} ="32",
(450,80) *+{\boxed{}} ="33",
(450,0) *+{\cdots} ="34",
(450,50) *+{\cdots} ="35",
(450,125) *+{\cdots} ="36",
"1", {\ar"8"},
"2", {\ar"5"},
"2", {\ar"8"},
"3", {\ar"6"},
"3", {\ar"8"},
"4", {\ar"7"},
"4", {\ar"8"},
"5", {\ar"10"},
"6", {\ar"11"},
"7", {\ar"12"},
"8", {\ar"9"},
"8", {\ar"10"},
"8", {\ar"11"},
"8", {\ar"12"},
"9", {\ar"16"},
"10", {\ar"13"},
"10", {\ar"16"},
"11", {\ar"14"},
"11", {\ar"16"},
"12", {\ar"15"},
"12", {\ar"16"},
"13", {\ar"18"},
"14", {\ar"19"},
"15", {\ar"20"},
"16", {\ar"17"},
"16", {\ar"18"},
"16", {\ar"19"},
"16", {\ar"20"},
"17", {\ar"24"},
"18", {\ar"21"},
"18", {\ar"24"},
"19", {\ar"22"},
"19", {\ar"24"},
"20", {\ar"23"},
"20", {\ar"24"},
"21", {\ar"26"},
"22", {\ar"27"},
"23", {\ar"28"},
"24", {\ar"25"},
"24", {\ar"26"},
"24", {\ar"27"},
"24", {\ar"28"},
"25", {\ar"32"},
"26", {\ar"29"},
"26", {\ar"32"},
"27", {\ar"30"},
"27", {\ar"32"},
"28", {\ar"31"},
"28", {\ar"32"},
"29", {\ar"34"},
"30", {\ar"35"},
"31", {\ar"36"},
"32", {\ar"33"},
"32", {\ar"34"},
"32", {\ar"35"},
"32", {\ar"36"},
\end{xy}
\]
\caption{The quivers $\Z \tilde{Q}$ associated with $A_2$ and $D_4$}
\label{fig:repetition-quivers-A2-D4}
\end{figure}

The {\em regular (or smooth) Nakajima category $\cR$} is the mesh
category $k(\Z\tilde{Q})$, where we take into account the presence of
frozen vertices by {\em only imposing} the mesh relations
$r_x$ associated with {\em non frozen vertices $x$}. The {\em singular Nakajima category
$\cs$} is the full subcategory of $\cR$ whose objects are the frozen 
vertices. In the main body of this article, we will work more 
generally with the quotient $\cs_C$ of $\cs$ associated
with a configuration of vertices $C$, cf.~section~\ref{ss:res-simple-RC-modules}.
This generality will in particular ensure that our results do contain
those of \cite{LeclercPlamondon12} as special cases.
For simplicity, in this description of the main results, we restrict ourselves to the case
where $\cs_C=\cs$. We write $\cR_0$ and $\cs_0$ for the sets of objects of
the categories $\cR$ and $\cs$. An {\em $\cs$-module} is a $k$-linear functor
$
M: \cs^{op} \to \Mod k
$,
where $\Mod k$ is the category of $k$-vector spaces
(cf.~section~\ref{ss:notations-and-recollections})
Let $w: \cs_0 \to \N$ be a {\em dimension vector}, i.e. a function with
finite support. The {\em affine graded quiver variety $\cm_0(w)$} is
the variety of $\cs$-modules $M$ such that $M(u) = k^{w(u)}$ for each
vertex $u$ in $\cs_0$. Notice that such a module is given by
the images of the morphisms of $\cs$ and that these have
to satisfy all the relations that hold in $\cs$. This shows
that the set $\cm_0(w)$ canonically identifies
with a Zariski closed subset of the finite-dimensional affine space
\[
\prod_{u_1, u_2} \Hom_k(\Hom_\cs(u_1,u_2), k^{w(u_2)\times w(u_1)}) \ko
\]
where the product ranges over all objects $u_1$, $u_2$ of $\cs$.
Thus, the set $\cm_0(w)$ becomes indeed canonically an affine variety. 
By Theorem~2.4 of \cite{LeclercPlamondon12}, based on 
\cite{Lusztig98a} \cite{LebruynProcesi90}, this definition of $\cm_0(w)$ is equivalent
to Nakajima's original definition in \cite{Nakajima01} when
$Q$ is a Dynkin quiver. The proof of   \cite{LeclercPlamondon12}
also shows that when $Q$ is bipartite (each vertex is a source or
a sink), our definition of $\cm_0(w)$ agrees with Nakajima's in
\cite{Nakajima11} and when $Q$ is an arbitrary acyclic quiver
with Kimura--Qin's in \cite{KimuraQin12}.

Neither the original definition of $\cm_0(w)$ nor the above variant
are very explicit. However, we can make the above definition
more explicit by describing the category $\cs$ by its quiver $Q_\cs$
with an admissible set of relations,
cf.~\cite[Ch.~8]{GabrielRoiter92} \cite[II.3]{AssemSimsonSkowronski06}.
Since the objects of $\cs$ are pairwise non isomorphic, we can identify
the set of vertices of $Q_\cs$ with $\cs_0$ and then the number of
arrows from $\sigma(y)$ to $\sigma(x)$ in $Q_\cs$ equals
\[
\dim \Ext^1_{\cs}(S_{\sigma(x)}, S_{\sigma(y)}) \ko
\]
where $S_{\sigma(x)}$ is the simple module associated with $\sigma(x)$.
Moreover, the number of relations from $\sigma(y)$ to $\sigma(x)$
equals
\[
\dim \Ext^2_{\cs}(S_{\sigma(x)}, S_{\sigma(y)}).
\]

\begin{theorem}[Cor.~ \ref{cor:ext-between-simples-in-sc}]
\label{thm:Ext-computation}
For each integer $p\geq 1$ and all vertices $x$, $y$ of $\Z Q$, we have
a canonical isomorphism
\[
\Ext^p_{\cs}(S_{\sigma(x)}, S_{\sigma(y)}) \iso \Hom_{\cd_Q}(H(x), \Sigma^p H(y)) \ko
\]
where $H$ is Happel's embedding (Theorem~\ref{thm:Happel}). These
spaces vanish if no connected component of $Q$ is a Dynkin quiver and $p\geq 2$.
\end{theorem}
Thanks to the theorem and to formula~(\ref{eq:Auslander-Reiten-formula}), 
we find that the
number of arrows, respectively minimal relations, from $\sigma(x)$ to 
$\sigma(y)$ equals
\[
\dim \Hom_{\cd_Q}(H(x), \tau H(y)) \mbox{ respectively } 
\dim \Hom_{\cd_Q}(H(x), \Sigma^{-1} \tau H(y)).
\]
It is not hard to see that this last dimension vanishes if no connected 
component of $Q$ is a Dynkin quiver. Thus, we obtain the following corollary.
\begin{corollary} If $Q$ is connected and not a Dynkin quiver, 
then for each dimension vector $w$,
the graded affine quiver variety $\cm_0(w)$ is isomorphic to an
affine space.
\end{corollary}
Let us consider two examples of Dynkin quivers:
For the quiver $Q: 1\to 2$, we find that $Q_\cs$ is the quiver
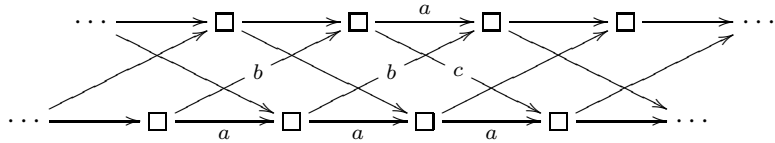
\[
\begin{xy} 0;<0.5pt,0pt>:<0pt,-0.5pt>:: 
(-50,0) *+{\cdots} ="-1",
(-100,75) *+{\cdots}="-2",
(0,75) *+{\boxed{}} ="0",
(50,0) *+{\boxed{}} ="1",
(100,75) *+{\boxed{}} ="2",
(150,0) *+{\boxed{}} ="3",
(200,75) *+{\boxed{}} ="4",
(250,0) *+{\boxed{}} ="5",
(300,75) *+{\boxed{}} ="6",
(350,0) *+{\boxed{}} ="7",
(400,75) *+{\cdots} ="8",
(450,0) *+{\cdots} ="9",
"-1", {\ar "1"},
"-1", {\ar "2"},
"-2", {\ar "0"},
"-2", {\ar "1"},
"0", {\ar_a "2"},
"0", {\ar|-*+{{\scriptstyle b}} "3"},
"1", {\ar"3"},
"1", {\ar"4"},
"2", {\ar_a "4"},
"2", {\ar|-*+{{\scriptstyle b}} "5"},
"3", {\ar^a "5"},
"3", {\ar|-*+{{\scriptstyle c}} "6"},
"4", {\ar_a "6"},
"4", {\ar"7"},
"5", {\ar"7"},
"5", {\ar"8"},
"6", {\ar"8"},
"6", {\ar"9"},
"7", {\ar"9"},
\end{xy}
\]
and that $\cs$ is isomorphic to the path category of $Q_\cs$ modulo the ideal
generated by all relations of the form $ab-ba$ and $a^3-cb$ (we denote
all horizontal arrows by $a$, all rising arrows by $b$ and all 
descending arrows by $c$). This example is deceptively simple.
The great complexity of the category $\cs$ becomes
visible when we look at the quiver of $\cs$ in the case
of $D_4$. In the following drawing, we only depict
the arrows which start at the leftmost vertex on
each row.
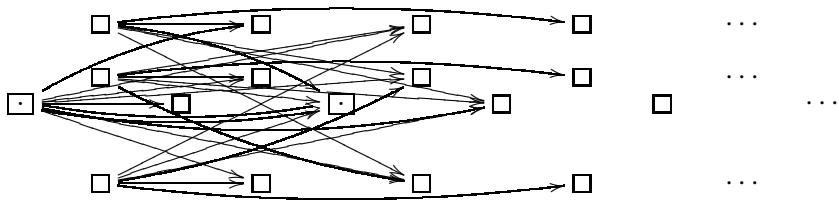
\[
\begin{xy} 0;<0.4pt,0pt>:<0pt,-0.4pt>:: 
(0,75) *+{\boxed{.}} ="0",
(75,0) *+{\boxed{}} ="1",
(75,50) *+{\boxed{}} ="2",
(75,150) *+{\boxed{}} ="3",
(150,75) *+{\boxed{}} ="4",
(225,0) *+{\boxed{}} ="5",
(225,50) *+{\boxed{}} ="6",
(225,150) *+{\boxed{}} ="7",
(300,75) *+{\boxed{.}} ="8",
(375,0) *+{\boxed{}} ="9",
(375,50) *+{\boxed{}} ="10",
(375,150) *+{\boxed{}} ="11",
(450,75) *+{\boxed{}} ="12",
(525,0) *+{\boxed{}} ="13",
(525,50) *+{\boxed{}} ="14",
(525,150) *+{\boxed{}} ="15",
(600,75) *+{\boxed{}} ="16",
(675,0) *+{\cdots} ="17",
(675,50) *+{\cdots} ="18",
(675,150) *+{\cdots} ="19",
(750,75) *+{\cdots} ="20",
"0", {\ar"4"},
"0", {\ar@/^/"5"},
"0", {\ar"6"},
"0", {\ar"7"},
"0", {\ar@{=>}@/_/"8"},
"0", {\ar "9"},
"0", {\ar"10"},
"0", {\ar"11"},
"0", {\ar@/_0.35cm/"12"},
"1", {\ar"5"},
"1", {\ar@/^/"8"},
"1", {\ar"10"},
"1", {\ar"11"},
"1", {\ar"12"},
"1", {\ar@/^/"13"},
"2", {\ar"6"},
"2", {\ar"8"},
"2", {\ar"9"},
"2", {\ar@/_/"11"},
"2", {\ar"12"},
"2", {\ar@/^/"14"},
"3", {\ar"7"},
"3", {\ar"8"},
"3", {\ar"9"},
"3", {\ar@/_/"10"},
"3", {\ar"12"},
"3", {\ar@/_/"15"},
\end{xy}
\]
Thus, the complete quiver is obtained from the one
displayed by adding all translates of the indicated arrows.
Notice that there is a double arrow between the
two vertices marked by a dotted box. This implies
that the
variety of representations with dimension vector
$(d_1, d_2)$ of the Kronecker quiver 
$\xymatrix{ 1 \ar@<0.5ex>[r] \ar@<-0.5ex>[r] & 2}$
is isomorphic to the graded affine quiver variety
$\cm_0(w)$ of type $D_4$ with dimension vector $w$
such that 
$w(\sigma(x))=d_1$, $w(\sigma(\tau^{-2}(x))=d_2$
for $x=(0,1)$ and $w(y)=0$ for all other frozen vertices $y$.
The stratification of this variety given by the orbits of
the base change group $GL_{d_1} \times GL_{d_2}$ 
has infinitely many strata already for $(d_1, d_2)=(1,1)$.
On the other hand, the Nakajima stratification, which
we will recall in the next section, always has finitely
many strata.

\subsection{Stratification} \label{ss:stratification}
We keep the assumptions
of section~\ref{ss:graded-quiver-varieties}. In particular,
$Q$ is an acyclic quiver and $\cR$ and $\cs$ are the
associated regular and singular Nakajima categories.
Let $v: \cR_0 \setminus \cs_0 \to \N$ and
$w: \cs_0 \to \N$ be dimension vectors. Let
$\tilde{\cm}(v,w)$ be the set of $\cR$-modules $M$ such
that $M(x) = k^{v(x)}$, $M(\sigma(x))= k^{w(\sigma(x))}$
for all $x\in \Z Q_0$ and that $M$ is {\em stable}, i.e.~we
have $\Hom_\cR(S_x, M)=0$ for each simple module
$S_x$ associated with a non frozen vertex
$x \in \Z Q_0$. Equivalently, $M$ does not contain
any non zero submodule supported only on non
frozen vertices. Let $G_v$ be the product of
the groups $\GL(k^{v(x)})$, where $x$ runs through the
non frozen vertices. By base change in the spaces
$k^{v(x)}$, the group $G_v$ acts freely on the
set $\tilde{\cm}(v,w)$. The {\em graded quiver
variety $\cm(v,w)$} is the quotient 
$\tilde{\cm}(v,w)/G_v$. For this definition and the following facts, we
refer to Nakajima's work \cite{Nakajima01} \cite{Nakajima11} for the
case where $Q$ is Dynkin or bipartite and to Qin \cite{Qin12} \cite{Qin12a} 
and Kimura-Qin \cite{KimuraQin12} for the extension to the case of an 
arbitrary acyclic quiver $Q$. The set  $\cm(v,w)$ canonically becomes a smooth quasi-projective variety and the projection map
\[
\pi : \cm(v,w) \to \cm_0(w)
\]
taking an $\cR$-module $M$ to its restriction
$M|_\cs$ is a proper map. Moreover, when $v$ varies, 
the graded affine quiver variety $\cm_0(w)$ is
stratified by the images of the non empty ones among the open subsets
$\cm^{reg}(v,w)\subset \cm(v,w)$ formed by the classes of the
modules $M \in \cm(v,w)$ which, in addition, are {\em co-stable}, i.e.
we have $\Hom_{\cR}(M,S_x)=0$ for each non
frozen vertex $x$ (by Prop.~4.1.3.8 of \cite{Qin12}, this is equivalent
to Nakajima's original description). The morphism $\pi$ induces
an isomorphism of each $\cm^{reg}(v,w)$ onto its
image in $\cm_0(w)$.

Recall that a {\em $\delta$-functor} from an abelian to a
triangulated category is (roughly) an additive functor 
transforming short exact sequences into triangles, cf. e.~g. \cite{Keller91}.
If no connected component of $Q$ is a Dynkin quiver,
let $\cv$ denote the additive subcategory of $\cd_Q$
whose indecomposable objects are the sums of
objects in the image of Happel's embedding. The
category $\cv$ becomes exact when endowed with
all the sequences giving rise to triangles in $\cd_Q$.

\begin{theorem}[sections~\ref{ss:construction-of-Phi}, \ref{ss:characterization-of-the-strata}]
\label{thm:strata}
There is a canonical $\delta$-functor
\[
\Phi: \mod \cs \to \cd_Q
\]
taking the simple module $S_{\sigma(x)}$ associated
with $x\in \Z Q_0$ to $H(x)$
(cf. Theorem~\ref{thm:Happel}) and such that
two modules $M_1$, $M_2$ belonging to $\cm_0(w)$ lie
in the same stratum if and only if $\Phi(M_1)$ is
isomorphic to $\Phi(M_2)$ in the derived category $\cd_Q$.
Moreover, if no connected component of $Q$ is a Dynkin
quiver, then $\Phi$ arises from an exact functor $\mod \cs \to \cv$.
\end{theorem}

The theorem is inspired by results obtained for Dynkin
quivers and particular choices of $w$ by Hernandez--Leclerc
\cite{HernandezLeclerc11} and by Leclerc--Plamondon
\cite{LeclercPlamondon12}. It suggests that the varieties
$\cm_0(w)$ should be related to the moduli stack
of objects of $\cd_Q$ introduced and studied
by To\"en--Vaqui\'e \cite{ToenVaquie07}. The following
theorem further underlines the geometric relevance
of the derived category.

\begin{theorem}[section~\ref{ss:degeneration-order}] 
\label{thm:degeneration-order}
Under the bijection between strata of $\cm_0(w)$ and 
isomorphism classes in its image under $\Phi$, the degeneration 
order among strata corresponds to the degeneration order of 
Jensen--Su--Zimmermann \cite{JensenSuZimmermann05a} 
among isomorphism classes in the derived category $\cd_Q$.
\end{theorem}

Note that for Dynkin quivers $Q$, the degeneration order
on strata of LP-varieties coincides with the degeneration order on
orbits in the representation spaces of the repetitive algebra
and also with the $\Hom$-order on isomorphism classes of
representations of the repetitive algebra, cf. Remark~3.15
of \cite{LeclercPlamondon12}.

Now consider the projection $\pi$ as a morphism
\[
\coprod_v \cm(v,w) \to \cm_0(w).
\]
The following theorem is a consequence of
Nakajima's slice theorem (section~3.3 of \cite{Nakajima01}
and section~2.4 of \cite{KimuraQin12}):

\begin{theorem}[section~\ref{ss:description-of-the-fibres}]
\label{thm:description-of-the-fibres}
For each module $M\in \cm_0(w)$, 
the fiber $\pi^{-1}(\{M\})$ is homeomorphic to 
the Grassmannian of $\cd_Q$-submodules of the
right $\cd_Q$-module
\[
D\Hom_{\cd_Q}(\Phi(M),?) : \cd_Q^{op} \to \mod k.
\]
\end{theorem}
Notice that each fiber contains a distinguished point: the
zero submodule. It corresponds to the pre-image of
$M$ under the isomorphism induced by $\pi$ from a
suitable $\cm^{reg}(v,w)$ onto the unique stratum containing $M$.

\subsection{Description of $\Phi$ via Kan extensions} 
\label{ss:descr-Phi-Kan-extensions}
Recall that a $k$-category is a category whose morphism spaces carry
$k$-vector space structures such that the composition is
bilinear. For a $k$-category $\cc$,
let $\Mod(\cc)$ denote the category of all {\em right $\cc$-modules},
i.e.~all $k$-linear functors $M: \cc^{op} \to \Mod k$, 
cf.~section~\ref{ss:notations-and-recollections}.

The inclusion $\cs \to \cR$ yields the restriction
functor $\res: \Mod(\cR) \to \Mod(\cs)$.  This functor admits a left adjoint $K_L$ and a
right adjoint $K_R$: the {\em left} and the {\em right Kan extension},
cf.~\cite{MacLane98}:
\[
\xymatrix{
\Mod(\cR) \ar[d]|-*+{{\scriptstyle \res}} \\
\Mod(\cs) \ar@<2ex>[u]^{K_L} \ar@<-2ex>[u]_{K_R}
}
\]
As we will see, they have simple and concrete descriptions. 
Both Kan extensions are fully faithful (and so $\res$ is 
a localization of abelian categories in the sense of \cite{Gabriel62}). 
They are linked by a canonical morphism
\begin{equation} \label{eq:KLtoKR}
\can: K_L \to K_R.
\end{equation}
By definition, the {\em intermediate Kan extension $K_{LR}$} is
its image, so that we have canonical morphisms
\begin{equation} \label{eq:KLtoKLRtoKR}
\xymatrix{K_L \ar@{->>}[r] &  K_{LR}\; \ar@{>->}[r] &  K_R.}
\end{equation}
The functor $K_{LR}$ restricted to certain subcategories 
plays an important role in \cite{CerulliFeiginReineke12}.
For special vectors $w$, the following proposition 
follows from section~3.3 of \cite{LeclercPlamondon12}. 

\begin{proposition}[section~\ref{ss:description-of-the-strata}]
 \label{prop:description-strata}
Let $w: \cs_0 \to \N$ be a dimension vector and
let $M\in \cm_0(w)$. Then the module $K_{LR}(M)$ is both stable and
co-stable and thus yields a point $\tilde{M}$ in $\cm^{reg}(v,w)$ for a suitable
$v$. The unique stratum containing $M$ is $\pi(\cm^{reg}(v,w))$ and $\tilde{M}$ is
the unique pre-image of $M$ under $\pi: \cm^{reg}(v,w) \to \cm_0(w)$.
\end{proposition}

It is not hard to check that $K_{LR}$ is in fact an equivalence from
$\Mod(\cs)$ onto the full subcategory of $\Mod(\cR)$ whose
objects are the modules which are both stable and co-stable.
The geometric meaning of the functor taking a stable $\cR$-module
$L$ to $K_{LR} (\res(L))$ is given by the following proposition,
which is essentially implicit in Nakajima's work \cite{Nakajima01}.

\begin{proposition}[section~\ref{ss:intermediate-extensions-closed-orbits}]
\label{prop:intermediate-extension-closed-orbit}
If $L$ is a stable $\cR$-module belonging to $\cm(v,w)$
and $K_{LR}(\res L)$ is of dimension vector $(v^0,w)$, then the unique
closed $G_v$-orbit in the closure of $G_v L$ in the affine variety 
$\rep(\cR^{op}, v,w)$ of representations of $\cR^{op}$ of dimension vector $(v,w)$ is
that of $K_{LR}(\res L) \oplus S$, where $S$ is the semi-simple
$k(\Z Q)$-module of dimension vector $v-v^0$.
\end{proposition}

By applying the proposition
to $L=K_{LR}(M)$ for an $\cs$-module $M$ (notice that $\res(K_{LR}(M))$
identifies with $M$)
we see in particular that the $G_v$-orbit of $K_{LR}(M)$ is closed
in the variety $\rep(\cR^{op}, v,w)$.

For each $\cs$-module $M$, the morphisms
$K_L(M) \to K_{LR}(M) \to K_R(M)$ become invertible
when restricted to $\cs$. 
Thus, the modules $CK(M)$ and $KK(M)$ defined by
\begin{align*}
KK(M) &=\ker(K_L(M) \to K_{LR}(M)) \mbox{ and } \\
CK(M) &=\cok(K_{LR}(M) \to K_R(M))
\end{align*}
vanish on $\cs$. Now we have an obvious isomorphism 
$\cR/\langle\cs \rangle \iso k(\Z Q)$, where $\langle\cs \rangle$ denotes
the ideal generated by the identical morphisms of $\cs$. Therefore,
we may view $CK(M)$ and $KK(M)$ as $k(\Z Q)$--modules. The
following proposition shows in particular that these modules are
injective respectively projective and that $KK$ and $CK$
determine $\Phi$.

\begin{proposition} [section~\ref{ss:description-of-Phi-via-Kan-extensions}]
\label{prop:description-KK-CK}
For $M\in\mod(\cs)$, we have functorial isomorphisms 
of $k(\Z Q)$-modules
\[
KK(M) = \Hom_{\cd_Q}(H(?), \tau\Phi(M)) 
\mbox{ and } 
CK(M) =  D\Hom_{\cd_Q}(\Phi(M),H(?))  \ko
\]
where $H$ is Happel's embedding (Theorem~\ref{thm:Happel}).
\end{proposition}

\subsection{Gorenstein homological algebra}
\label{ss:Gorenstein-homological-algebra}
Let us assume that $Q$ is connected. The construction
of the stratifying functor $\Phi$ of Theorem~\ref{thm:strata}
is given in section~\ref{ss:construction-of-Phi}.
The proof of its exactness properties is quite different
depending on whether $Q$ is a Dynkin quiver
or not: If $Q$ is not Dynkin, we give a direct argument in
section~\ref{ss:stratifying-functor-non-Dynkin-case}. 
In the Dynkin case, we use Gorenstein homological algebra
(cf. section~\ref{s:stratifying-functor-Dynkin-case}): 
Let us assume that $Q$ is a Dynkin quiver and
$\cs$ the associated singular Nakajima category.
Recall that an $\cs$-module $M$ is {\em finitely presented} if
there is a projective presentation
\[
P_1 \to P_0 \to M \to 0
\]
with finitely generated projective modules $P_0$ and $P_1$.
An $\cs$-module $M$ is {\em right-bounded} if, for all $p\gg 0$,
the space $M(i,p)$ vanishes
for all $i\in Q_0$; it is  {\em pointwise finite-dimensional}
if all the spaces $M(u)$, $u\in \cs_0$, are finite-dimensional.

\begin{proposition}[sections~\ref{ss:Gorenstein-property}, \ref{ss:coherence}]
\label{prop:properties-singular-category}
\begin{itemize}
\item[a)]
The category $\cs$ is coherent, i.e.~its category of finitely presented modules is abelian.
\item[b)] The category $\cs$ is weakly Gorenstein of dimension $1$ in the sense that,
for all $p>1$, we have 
\[
\Ext^p_\cs(M,P)=0
\]
for each right-bounded pointwise finite-dimensional 
module $M$ and each finitely generated projective module $P$.
\end{itemize}
\end{proposition}

An $\cs$-module $M$ is {\em Gorenstein-projective} if
$\Ext^p_\cs(M,P)=0$ for all $p>1$ and all finitely generated
projective $\cs$-modules $P$. Let $\gpr(\cs)$ be the
category of finitely presented Gorenstein-projective
$\cs$-modules. For each finitely generated $\cs$-module
$M$, let $\Omega M$ be the kernel of a surjective
morphism $P \to M$, where $P$ is finitely generated
projective. 

\begin{theorem}[sections~\ref{ss:two-Frobenius-categories} and
\ref{ss:link-to-the-derived-category}]
\label{thm:stable-category-of-Gorenstein-projectives} 
The category $\gpr(\cs)$ is a
Frobenius category. There is a canonical
equivalence from its stable category $\ul{\gpr}(\cs)$ to $\cd_Q$
sending $\Omega S_{\sigma(x)}$ to $H(x)$
(Theorem~\ref{thm:Happel}) for each $x\in\Z Q_0$.
\end{theorem}

Now the functor $\Phi$ is obtained as the composition
\[
\xymatrix{
\mod(\cs) \ar[r]^{\Omega} & \ul{\gpr}(\cs) \ar[r]^-{\sim} 
& \cd_Q \ko
}
\]
which shows in particular that it is a $\delta$-functor.

Let $\proj(\cR)$ denote the category of the finitely
generated projective $\cR$-modules. The following
theorem allows us to view the regular category $\cR$
as an Auslander category for the Gorenstein projective
$\cs$-modules.

\begin{theorem}[section~\ref{ss:regular-category-as-Auslander-category}]
\label{thm:regular-category-as-Auslander-category}
The restriction functor induces
equivalences $\proj(\cR) \to \gpr(\cs)$ and
$\inj(\cR) \to \gin(\cs)$. It yields
isomorphisms from the quiver $\Z \tilde{Q}$ onto
the Auslander--Reiten quivers of $\gpr(\cs)$ and $\gin(\cs)$
so that the frozen vertices correspond to the projective-injective
vertices.
\end{theorem}

In particular, we obtain that $\proj(\cR)$ admits a
natural structure of standard (in the sense of section~2.3,
page~63 of \cite{Ringel84}) Frobenius category whose projectives
are the finite direct sums of indecomposable projectives
associated with the frozen vertices of $\Z \tilde{Q}$ and
where each mesh ending in a non frozen vertex yields an
Auslander--Reiten conflation. More generally,
in section~\ref{s:stratifying-functor-Dynkin-case},
we will prove the above results (except coherence) for the quotients
$\cR \to \cR_C$ and $\cs\to \cs_C$ associated 
with suitable configurations $C$, cf.~section~\ref{ss:res-simple-RC-modules}.
In the philosophy of \cite{IyamaKalckWemyssYang12}, the 
Frobenius category $\ce=\proj(\cR_C)$ `admits a resolution', namely
itself, and so one expects its category of projectives $\proj(\cs_C)$ to be
Gorenstein and the category itself to be equivalent to
the category Gorenstein-projective modules over $\proj(\cs_C)$.
Technically, the categories we consider do not quite fit into the
framework of [loc. cit.] but the philosophy of that
paper is compatible with our findings.

\section{Homological properties of the Nakajima categories}
\label{s:Homological-properties-of-the-Nakajima-categories}

\subsection{Notations and recollections} \label{ss:notations-and-recollections}
Let $k$ be a field and $\Mod k$ the category of $k$-vector spaces. 
Recall that a {\em $k$-category} is a category whose
morphism spaces are endowed with $k$-vector space structures such
that the composition is bilinear. Let $\cc$ be a $k$-category. We denote
by $\Mod(\cc)$ the category of {\em right $\cc$-modules}, i.e.~$k$-linear functors 
$\cc^{op} \to \Mod(k)$. For each object $x$ of $\cc$, we have the {\em free
module}
\[
x^{\wedge}=x^{\wedge}_\cc = \cc(?,x): \cc^{op} \to \Mod k
\]
and the {\em cofree module}
\[
x^{\vee}= x^\vee_\cc = D (\cc(x,?)): \cc^{op} \to \Mod k.
\]
Here, we write $\cc(u,v)$ for the space of morphisms $\Hom_\cc(u,v)$ and
$D$ for the duality over the ground field $k$. For each object $x$ of $\cc$ and
each $\cc$-module $M$, we have canonical isomorphisms
\begin{equation} \label{eq:Yoneda}
\Hom(x^\wedge,M) = M(x) \quad\mbox{and}\quad
\Hom(M,x^\vee) = D(M(x)).
\end{equation}
In particular, the module $x^\wedge$ is projective and $x^\vee$ is injective.
A module is {\em finitely generated} if it is a quotient of a finite direct sum
of modules $x^\wedge$; it is {\em finitely cogenerated} if it is a submodule of
a finite direct sum of modules $x^\vee$.
If $x$ is an object of $\cc$ whose
endo\-morphism algebra is local, then the free module $x^\wedge$ admits
a unique {\em simple quotient $S_x$}, which is also the unique simple
submodule of $x^\vee$.
By Kaplansky's theorem \cite{Kaplansky58}, if the endomorphism ring of
each object $x$ of $\cc$ is local, each projective module over $\cc$ is
a direct sum of free modules $x^\wedge$. 

Let us make the following assumptions on $\cc$.
\begin{assumption} \label{as:directed-category}
\begin{itemize}
\item[a)] The morphism spaces of $\cc$ are finite-dimensional and
\item[b)] the category $\cc$ is {\em directed}, i.e. the endomorphism algebra 
of each object is $k$ and $\cc$  is endowed with an order relation such 
that $\cc(x,y)\neq 0$ implies $x\leq y$.
\end{itemize} 
\end{assumption}
A $\cc$-module $M$ is {\em pointwise
finite-dimensional} if $M(x)$ is finite-dimensional for each object $x$ of $\cc$.
It is {\em right bounded} if there is a finite set of objects $E$ such that each
object $x$ with $M(x)\neq 0$ is less or equal to an object of $E$.
For example, the modules $x^\wedge$ are pointwise finite-dimensional
and right bounded. Let $M$ be a pointwise finite-dimensional, right bounded module.
Then  $M$ admits
a projective cover $P \to M$ by a (usually infinite)
coproduct $P$ of free modules $x^\wedge$
and the multiplicity of $x^\wedge$ in $P$ is
finite and equals the dimension of $\Hom(M, S_x)$. Moreover, the kernel
of $P\to M$ is again right bounded and pointwise finite-dimensional.
Thus, the module $M$ admits a minimal projective resolution
\[
\ldots \to P_1 \to P_0 \to M \to 0
\]
where each object $P_i$ is a coproduct of free modules $x^\wedge$ 
and the multiplicity of $x^\wedge$ in $P_i$ equals the dimension of
$\Ext^i(M,S_x)$.

\subsection{Resolutions for the simple $\cR_C$-modules} 
\label{ss:res-simple-RC-modules}
Let $Q$ be a connected acyclic quiver and $C$ a subset of the set 
of vertices of the repetition quiver $\Z Q$. Let $\cR_C$ be the
quotient of $\cR$ by the ideal generated by the identities of the
frozen vertices not belonging to $\sigma^{-1}(C)$ and let $\cs_C$ be the
full subcategory of $\cR_C$ formed by the vertices in $\sigma^{-1}(C)$.
We make the following assumption on $C$.
\begin{assumption} \label{main-assumption}
For each non frozen vertex $x$ of $\Z \tilde{Q}$, the sequences
\begin{equation} \label{eq:left-exact-sequences}
0 \to \cR_C(?,x) \to \bigoplus_{x \to y} \cR_C(?,y) \quad \mbox{and}\quad
0 \to \cR_C(x,?) \to \bigoplus_{y \to x} \cR_C(y,?) 
\end{equation}
are exact, where the sums range over all arrows of $\Z \tilde{Q}$ whose
source (respectively, target) is $x$.
\end{assumption}
The assumption holds, for example, if $C$ is the set of all vertices of $\Z Q$.
It also holds in the following situation: Assume that $\ce$ is a 
$\Hom$-finite exact Krull--Schmidt category which is standard 
(in the sense of section~2.3, page~63 of \cite{Ringel84}) and
whose Auslander--Reiten quiver is the full subquiver of $\Z \tilde{Q}$ formed
by the non frozen vertices and the vertices $\sigma^{-1}(c)$, $c\in C$, where the latter
correspond to the projective indecomposables of $\ce$. Then the
sequences (\ref{eq:left-exact-sequences}) are associated with
Auslander--Reiten conflations of $\ce$ and hence are exact.
For example, one can take $\ce$ to be the category of
finite-dimensional modules over the repetitive algebra of
an iterated tilted algebra $B$ of Dynkin type. The case where
$B$ itself is the path algebra of a Dynkin quiver was considered
by Leclerc--Plamondon \cite{LeclercPlamondon12}.


In fact, as we will see in 
Theorem~\ref{thm:regular-category-as-Auslander-category-for-config}, 
when the assumption holds, the given
set $C$ always comes from the choice of a $\Hom$-finite exact
Krull--Schmidt category which is moreover standard and whose stable 
Auslander--Reiten quiver is $\Z Q$. 

Another sufficient condition for the assumption to hold
is due to Iyama: According to 7.4 (2) of \cite{Iyama05b}, the
assumption holds if for each vertex $x$ of $\Z Q$, there is a
vertex $c$ in $C$ such that there is a non zero morphism
from $x$ to $c$ in the mesh-category $\cR_C$. Notice that
for this, it is sufficient that the following condition (R) holds:
\begin{itemize}
\item[(R)] for each vertex $x$ of $\Z Q$, there is a vertex
$c$ in $C$ such that the space of morphisms $k(\Z Q)(x,c)$
in the mesh category of $\Z Q$ does not vanish.
\end{itemize}
This is the first condition 
which Riedtmann imposed on the `combinatorial configurations' 
in her sense, cf.~Definition~2.3 of \cite{Riedtmann80a} 
and \cite{BretscherLaeserRiedtmann81}. 

Returning to the general setup we have the following lemma.

\begin{lemma} \label{lemma:resolutions}
For each (non frozen) vertex $x$ of $\Z Q$, we have (co-)resolutions
of simple $\cR_C$-modules
\begin{itemize}
\item[a)] $0 \to \tau(x)^\wedge \to \sigma(x)^\wedge \to S_{\sigma(x)} \to 0$,
\item[b)] $0 \to S_{\sigma(x)} \to \sigma(x)^\vee \to x^\vee \to 0$,
\item[c)] $0 \to \tau(x)^\wedge \to \bigoplus_{y\to x} y^\wedge \to x^\wedge \to S_x \to 0$,
\item[d)] $0 \to S_x \to x^\vee \to \bigoplus_{x\to y} y^\vee \to \tau^{-1}(x)^\vee \to 0$,
\end{itemize}
where in c) and d), the sum ranges over all arrows $y \to x$ with target $x$
in the quiver  $\Z \tilde{Q}$, respectively all arrows $x\to y$ with source $x$.
\end{lemma}
The proof is an exercise. By the isomorphisms~(\ref{eq:Yoneda}), we 
immediately obtain the following corollary.

\begin{corollary} \label{cor:RHom} For each $\cR_C$-module $M$ and
each vertex $x$ of $\Z Q_0$, we have canonical isomorphisms 
in the derived category of vector spaces, where the
first term on the right is always in degree $0$:
\begin{itemize}
\item[a)] $\RHom(M, S_{\sigma(x)}) = (DM(\sigma(x)) \to DM(x))$,
\item[b)] $\RHom(S_{\sigma(x)},M)= (M(\sigma(x)) \to M(\tau(x)))$,
\item[c)] $\RHom(M, S_x) = (DM(x) \to \bigoplus_{x\to y} DM(y) \to DM(\tau^{-1}(x)))$,
\item[d)] $\RHom(S_x, M) = (M(x) \to \bigoplus_{y\to x} M(y) \to M(\tau(x)))$.
\item[e)] In particular, we have a canonical isomorphism
\begin{equation} \label{eq:duality}
D\RHom(S_x, M) = \RHom(M, \Sigma^2 S_{\tau(x)}) \ko
\end{equation}
where $\Sigma$ denotes the suspension functor.
\end{itemize}
\end{corollary}

Let $\langle C\rangle$ denote the ideal of $\cR_C$ generated by
the identities of the vertices in $\sigma^{-1}(C)$. By Happel's theorem (\ref{thm:Happel}),
we have a fully faithful embedding
\[
H: \cR_C/\langle C \rangle \to \cd_Q \ko
\]
which is an equivalence if and only if $Q$ is a Dynkin quiver. If
$Q$ is Dynkin, let $\Sigma$ be the unique bijection of the
vertices of $\Z Q$ such that
\[
H(\Sigma x) = \Sigma H(x).
\]
If $Q$ is arbitrary acyclic, for each non frozen vertex $x\in\Z Q_0$, let
\begin{align*}
x^\wedge_\cd & = (\cR_C/\langle C\rangle)(?, x) = \cd_Q(?, x) 
\quad\mbox{and}\quad \\
x^\vee_\cd &= D(\cR_C/\langle C \rangle)(x,?)) = D \cd_Q(x, ?) \ko
\end{align*}
where, for simplicity, we omit the Happel functor $H$ from the notations.
Moreover, put
\begin{equation} \label{eq:pc-ic}
P_C(x) = \bigoplus_{\sigma(y)\in C} \cd_Q(y,x) \ten \sigma(y)^\wedge 
\quad\mbox{and}\quad
I_C(x)= \prod_{\sigma^{-1}(y)\in C} D\cd_Q(x,y) \ten \sigma^{-1}(y)^\vee.
\end{equation}
\begin{theorem}  \label{thm:rc-resolutions}
\begin{itemize}
\item[a)] Suppose that $Q$ is a Dynkin quiver.
For each non frozen vertex $x\in \Z Q_0$, we have a resolution
of $\cR_C$-modules
\begin{equation} \label{eq:proj-res-xd}
0 \to (\Sigma^{-1} x)^\wedge \to P_C(x) \to x^\wedge \to x^\wedge_\cd \to 0
\end{equation}
and a coresolution
\begin{equation} \label{eq:inj-res-xd}
0 \to x^\vee_\cd \to x^\vee \to I_C(x) \to (\Sigma x)^\vee \to 0.
\end{equation}
\item[b)] Suppose that $Q$ is not a Dynkin quiver.
For each non frozen vertex $x\in \Z Q_0$, we have a resolution
of $\cR_C$-modules
\[
0 \to P_C(x) \to x^\wedge \to x^\wedge_\cd \to 0
\]
and a coresolution
\[
0 \to x^\vee_\cd \to x^\vee \to I_C(x)  \to 0.
\]
\end{itemize}
\end{theorem}

\begin{proof} Note that the category $\cR_C$ satisfies the assumptions~\ref{as:directed-category}. Thus, to check the claims, it suffices to compute the extensions between
the simple modules $S_u$, where $u$ is any vertex of $\Z \tilde{Q}$, 
and $x^\wedge_\cd$ respectively $x^\vee_\cd$. 
For this, we use Lemma~\ref{lemma:resolutions}. 
Let $y$ be a non frozen vertex. We have
\[
\RHom(x^\wedge_\cd, S_{\sigma(y)}) = 
\RHom(x^\wedge_\cd, \sigma(y)^\vee \to y^\vee) =
(0 \to D\cd_Q(y,x)).
\]
This yields the term $P_C(x)$ in the resolution~(\ref{eq:proj-res-xd}).
Similarly, we find
\begin{align*}
\RHom(x^\wedge_\cd, S_y) &= 
\RHom(x^\wedge,(y^\vee \to \bigoplus_{y\to z} z^\vee\to \tau^{-1}(y)^\vee)) \\
&= (D\cd_Q(y,x)\to \prod_{y\to z} D\cd_Q(z,x) \to D\cd_Q(\tau^{-1}(y), x)).
\end{align*}
This complex is also obtained by applying $\Hom(x^\wedge,?)$ to the
complex of $\cd_Q$-modules
\[
y_\cd^\vee \to \prod_{y\to z} z^\vee_\cd \to \tau^{-1}(y)^\vee_\cd
\]
which is associated with the Auslander--Reiten triangle
\[
y \to \bigoplus_{y\to z} z \to \tau^{-1}(y) \to \Sigma y
\]
of $\cd_Q$. Thus, we have an exact sequence of $\cd_Q$-modules
\[
0 \to S^\cd_y \to y_\cd^\vee \to \prod_{y\to z} z^\vee_\cd \to \tau^{-1}(y)^\vee_\cd
\to S^\cd_{\Sigma y} \to 0 \ko
\]
where $S^\cd_y$ is the simple $\cd_Q$-module associated with $y$.
It follows that the homology of $\RHom(x^\wedge_\cd, S_y)$ is given
by $S^\cd_y(x)$ in degree $0$ and $S^\cd_{\Sigma y}(x) = S^\cd_y(\Sigma^{-1}(x))$
in degree $2$. This yields the projective resolution in a). In the
non Dynkin case, no object $H(y)$, $y\in \Z Q_0$, is isomorphic to
$\Sigma^{-1} H(x)$. Thus, the homology of $\RHom(x^\wedge_\cd, S_y)$
in degree $2$ vanishes and we find the projective resolution in b). 
A similar argument yields the injective co-resolutions in a) and b).
\end{proof}

\subsection{Resolutions for the simple $\cs_C$-modules}
\label{ss:res-simple-SC-modules}
We keep the notations and assumptions of section~\ref{ss:res-simple-RC-modules}.
Notice that for each frozen vertex $\sigma(x)$, $x\in \Z Q_0$, the restriction of the
free $\cR_C$-module $\sigma(x)^\wedge_{\cR_C}$ to $\cs_C$ is the
free $\cs_C$-module $\sigma(x)^\wedge_{\cs_C}$ and similarly
for the co-free modules associated with the frozen vertices. In particular,
the restrictions  to $\cs_C$ of the modules $P_C(x)$ and $I_C(x)$ defined in
(\ref{eq:pc-ic}) are still projective respectively injective.
By abuse of notation, we denote the restricted modules by
the same symbols $P_C(x)$ and $I_C(x)$. 

\begin{theorem} \label{thm:res-simple-SC-modules}
Suppose that $Q$ is connected. Let $x$ be a vertex of $\Z Q$. 
\begin{itemize}
\item[a)] If $Q$ is a Dynkin quiver, the simple $\cs_C$-module
$S_{\sigma^{-1}(x)}$ admits a minimal projective resolution of the form
\[
\ldots \to P_C(\Sigma^{-2} x) \to 
P_C(\Sigma^{-1} x) \to P_C(x) \to \sigma^{-1}(x)^\wedge \to 
S_{\sigma^{-1}(x)} \to 0
\]
and the simple $\cs_C$-module $S_{\sigma(x)}$ admits a minimal
injective resolution of the form
\[
0 \to S_{\sigma(x)} \to \sigma(x)^\vee \to I_C(x) \to I_C(\Sigma x) \to
I_C(\Sigma^2 x) \to \ldots\ .
\]
\item[b)] If $Q$ is not a Dynkin quiver, the simple $\cs_C$-module
$S_{\sigma^{-1}(x)}$ admits a minimal projective resolution of the form
\[
0 \to P_C(x) \to \sigma^{-1}(x)^\wedge \to S_{\sigma^{-1}(x)} \to 0
\]
and the simple $\cs_C$-module $S_{\sigma(x)}$ admits a minimal
injective resolution of the form
\[
0 \to S_{\sigma(x)} \to \sigma(x)^\vee \to I_C(x)  \to 0.
\]
\end{itemize}
\end{theorem}

\begin{proof} 
Part a) of Lemma~\ref{lemma:resolutions} yields an exact sequence
of $\cR_C$-modules
\[
0 \to x^\wedge \to \sigma^{-1}(x)^\wedge \to S_{\sigma^{-1}(x)} \to 0.
\]
We restrict it to $\cs_C$ and now have to construct a minimal resolution
for $\res(x^\wedge)$, where $\res$ denotes the restriction functor. 
If $Q$ is not a Dynkin quiver, part b) of Theorem~\ref{thm:rc-resolutions} yields the exact sequence of $\cR_C$-modules
\[
0 \to  P_C(x) \to x^\wedge \to x^\wedge_\cd \to 0.
\]
Since the restriction of $x^\wedge_\cd$ to $\cs_C$ vanishes, we find
that $\res(x^\wedge)$ is isomorphic to the restriction of $P_C(x)$ to $\cs_C$,
which yields the projective resolution in b). If $Q$ is a Dynkin quiver, 
part a) of Theorem~\ref{thm:rc-resolutions} yields
the exact sequence of $\cR_C$-modules
\[
0 \to (\Sigma^{-1} x)^\wedge \to P_C(x) \to x^\wedge \to x^\wedge_\cd \to 0.
\]
Since the restriction $\res(x^\wedge_\cd)$ vanishes, we obtain a short
exact sequence
\[
0 \to \res((\Sigma^{-1} x)^\wedge) \to P_C(x) \to \res(x^\wedge)  \to 0
\]
and more generally an exact sequence
\[
0 \to \res((\Sigma^{-(p+1)} x)^\wedge) \to P_C(\Sigma^{-p}x) \to
 \res((\Sigma^{-p}x)^\wedge)  \to 0
\]
for each $p\geq 0$. We obtain the desired resolution by splicing these
sequences together. The construction of the injective co-resolutions is
analogous.
\end{proof}

\begin{corollary} \label{cor:ext-between-simples-in-sc}
Let $x$ and $y$ be vertices of $\Z Q$. For each $p\geq 1$, 
we have an isomorphism
\[
\Ext^p_{\cs_C}(S_{\sigma(x)}, S_{\sigma(y)}) = \cd_Q(H(x), \Sigma^p H(y)).
\]
If $Q$ is not a Dynkin quiver, these spaces vanish for all $p\geq 2$. 
\end{corollary}

\begin{proof} We use the minimal injective co-resolutions in 
Theorem~\ref{thm:res-simple-SC-modules}. Suppose that $Q$ is
a Dynkin quiver and $p\geq 1$. We have 
\begin{align*}
\Ext^p_{\cs_C}(S_{\sigma(x)}, S_{\sigma(y)}) &= 
	\Hom(S_{\sigma(x)}, I_C(\Sigma^{p-1}y))\\
&=\Hom(S_{\sigma(x)}, \prod D\cd_Q(\Sigma^{p-1}y,z)\ten \sigma^{-1}(z)^\vee).
\end{align*}
Now the space $\Hom(S_{\sigma(x)}, \sigma^{-1}(z)^\vee)$ vanishes unless
$\sigma(x) = \sigma^{-1}(z)$, that is $z=\sigma^2(x) = \tau(x)$. Hence we find
\begin{align*}
\Ext^p_{\cs_C}(S_{\sigma(x)}, S_{\sigma(y)}) &= D\cd_Q(\Sigma^{p-1} y, \tau(x))
=D\cd_Q(\Sigma^{p-1} y, \Sigma^{-1} Sx)\\
&= \cd_Q(x, \Sigma^p y).
\end{align*}
\end{proof}

\section{The stratifying functor $\Phi$}
\label{s:stratifying-functor}

\subsection{Construction of $\Phi$} \label{ss:construction-of-Phi}
Let $Q$ be a connected acyclic quiver and $C$ a subset of the set of frozen
vertices of the repetition quiver $\Z \tilde{Q}$ which satifies 
Assumption~\ref{main-assumption}. Notice that $\Mod(\cR_C)$ is a subcategory
of $\Mod(\cR)$ and similarly $\Mod(\cs_C)$ a subcategory of
$\Mod(\cs)$. Let $\res^C: \Mod(\cR_C) \to \Mod(\cs_C)$
be the restriction functor. Clearly, it is just the restriction of the 
functor $\res: \Mod(\cR) \to \Mod(\cs)$ to the subcategories under
consideration. The left and right adjoints $K_L$ and $K_R$ of
$\res$ take the subcategory $\Mod(\cs_C)$ of $\Mod(\cs)$ to
$\Mod(\cR_C)$ and thus induce left and right adjoints
$K_L^C$ and $K_R^C$ of $\res^C$ so that we have 
\[
\xymatrix{
\Mod(\cR_C) \ar[d]|-*+{{\scriptstyle \res^C}} \\
\Mod(\cs_C). \ar@<3ex>[u]^{K^C_L} \ar@<-3ex>[u]_{K_R^C}
}
\]
The functor $\res^C$ is still a localization of abelian categories in the sense of \cite{Gabriel62}.  {\em In the sequel, we will omit the exponents $C$ in the 
notation for the functors $K^C_L$ and $K_R^C$ and simply write $K_L$ and
$K_R$}. We have the canonical morphism 
\begin{equation} \label{eq:KCLtoKCR}
\can: K_L \to K_R.
\end{equation}
(which is just the restriction to $\Mod(\cs_C)$ of the canonical
morphism between the non restricted functors).
By definition, the {\em intermediate Kan extension $K_{LR}$} is
its image, the {\em functor $KK$} its kernel and the {\em functor $CK$}
its cokernel so that we have the following diagram of functors from
$\Mod(\cs_C)$ to $\Mod(\cR_C)$: 
\[
\xymatrix{0 \ar[r] & KK \ar[r] & K_L \ar[rr] \ar@{->>}[dr] & & K_R \ar[r] & CK \ar[r] & 0. \\
                           &               &                     &K_{LR} \ar@{>->}[ru] &   &                & }
\]
The {\em kernel $\cn$} of the functor $\res$ is the abelian subcategory
formed by the modules which vanish on the frozen vertices. Clearly,
the category $\cn$ is isomorphic to the category of modules over the
quotient $\cR_C/\langle \cs_C\rangle$ of $\cR_C$ by the ideal
generated by the identities of the objects of $\cs_C$. Notice
that this quotient is isomorphic to the mesh category $k(\Z Q)$.
We will identify
\[
\cn = \Mod(\cR_C/\langle \cs_C \rangle) = \Mod(k(\Z Q)).
\]
In Theorem~\ref{thm:representability} below, we will see that for each
finite-dimensional $\cs_C$-module $M$, the $k(\Z Q)$-module
$CK(M)$ is a finitely cogenerated injective module. Thus, there is
an object $\Phi(M)$ in the derived category $\cd_Q$ such that
\[
CK(M)(x) = D\Hom(\Phi(M), H(x))
\]
functorially in $x\in \Z Q$. Clearly, the map $M \mapsto \Phi(M)$ underlies
a {\em $k$-linear functor}
\[
\Phi: \mod(\cs_C) \to \cd_Q. 
\]
We call $\Phi$ the {\em stratifying functor} because of Theorem~\ref{thm:strata}.
Our construction does not make it clear which exactness properties
the stratifying functor has but we will show the
following theorem. Recall from section~\ref{ss:stratification}, that
if no connected component of $Q$ is a Dynkin quiver,
then $\cv$ denotes the additive subcategory of $\cd_Q$
whose indecomposable objects are the sums of
objects in the image of Happel's embedding. The
category $\cv$ becomes exact when endowed with
all the sequences giving rise to triangles in $\cd_Q$.

\begin{theorem} \label{thm:exactness-of-Phi}
\begin{itemize}
\item[a)] If no connected component of $Q$ is a Dynkin quiver,
then $\Phi$ is exact as a functor $\mod(\cs_C) \to \cv$.
\item[b)] If $Q$ is a Dynkin quiver, then $\Phi$ underlies a
$\delta$-functor $\mod(\cs_C) \to \cd_Q$.
\end{itemize}
\end{theorem}

We will prove part a) in section~\ref{ss:stratifying-functor-non-Dynkin-case}. 
Part b) will follow from the alternative construction of $\Phi$ via Gorenstein homological
algebra in section~\ref{ss:description-of-Phi-via-Kan-extensions}.

\subsection{The representability theorem} \label{ss:representability-theorem}
Our goal in this section is to prove Theorem~\ref{thm:representability}.

\begin{lemma} \label{lemma:characterization-image-KR-KL}
An $\cR_C$-module $M$ belongs to the image
of $K_R$, respectively $K_L$, if and only if, for each $N\in \cn$, we have
\[
\Hom(N,M)=0 \mbox{ and } \Ext^1(N,M)=0 \ko
\]
respectively 
\[
\Hom(M,N)=0 \mbox{ and } \Ext^1(M,N)=0.
\]
\end{lemma}

\begin{proof} This is a general characterization of the image of the adjoint
of a localization functor, cf. Lemme~1, page~370 of \cite{Gabriel62}.
\end{proof}

\begin{lemma} \label{lemma:Hom-Ext-KLR}
Let $M$ be an $\cs_C$-module and $N$ in $\cn$. We have
canonical isomorphisms
\begin{align*}
\Hom(N, CK(M)) &\iso \Ext^1(N, K_{LR}M) \mbox{ and } \\
\Hom(KK(M),N) & \iso \Ext^1(K_{LR} M,N).
\end{align*}
\end{lemma}

\begin{proof} The first isomorphism is obtained by applying the functor
$\Hom(N,?)$ to the exact sequence
\[
0 \to K_{LR}(M) \to K_R(M) \to CK(M) \to 0
\]
and using Lemma~\ref{lemma:characterization-image-KR-KL}. Dually,
one obtains the second isomorphism.
\end{proof}

We will see in Theorem~\ref{thm:representability} below that if $M$ is
a finite-dimensional $\cs_C$-module, then $KK(M)$ is projective
and $CK(M)$ is injective. The main step in the proof is the following
lemma.

\begin{lemma} \label{lemma:exactness}
Let $M$ be a finite-dimensional $\cs_C$-module and 
\begin{equation} \label{eq:ex-seq-N}
0 \to N' \to N \to N'' \to 0
\end{equation}
an exact sequence in $\cn$. 
\begin{itemize}
\item[a)] If $N'$ is right bounded and pointwise finite-dimensional, the left exact
functor $\Hom(KK(M),?)$ transforms (\ref{eq:ex-seq-N}) into an exact sequence.
\item[b)] If $N''$ is left bounded and pointwise finite-dimensional, the left exact
functor $\Hom(?, CK(M))$ transforms (\ref{eq:ex-seq-N}) into an exact sequence.
\end{itemize}
\end{lemma}

\begin{proof} We prove a), the proof of b) being dual.
By Lemma~\ref{lemma:Hom-Ext-KLR}, on the subcategory $\cn$,
the functor $\Hom(KK(M), ?)$ is isomorphic to $\Ext^1(K_{LR}(M),?)$. We have
an exact sequence
\begin{align*}
(K_{LR}(M), N'') &\to \mbox{}^1(K_{LR}(M), N') \to \mbox{}^1(K_{LR}(M),N) 
\to \mbox{}^1(K_{LR},N'') \\
& \to \mbox{}^2(K_{LR}(M), N')\ko
\end{align*}
where we abbreviate $\Hom(,)$ by $(,)$ and $\Ext^p(,)$ by $\mbox{}^p(,)$.
Here the group
\[
\Hom(K_{LR}(M),N'')
\]
vanishes since $K_{LR}(M)$ is a quotient of
$K_L(M)$. The claim will follow once we show that $\Ext^2(K_{LR}(M),N')$
vanishes. Since $N'$ is right bounded and pointwise finite-dimensional,
it is the inverse limit of a countable system $N'_p$, $p\in \N$, of finite-dimensional
$\cR_C$-modules. We have an exact sequence
\[
0 \to {\lim}^1 \Ex^1(K_{LR}(M), N'_p) \to \Ex^2(K_{LR}(M), N') \to
\lim \Ex^2(K_{LR}(M), N'_p) \to 0 \ko
\]
where $\Ex^i$ denotes $\Ext^i$ and ${\lim}^1$ the first right derived 
functor of the inverse
limit functor $\lim$, cf.~Application 3.5.10 in \cite{Weibel94}.
Since $K_{LR}(M)$ is finite-dimensional, it admits a finite resolution by
finitely generated projective $\cR_C$-modules. Now $\Hom(P, N'_p)$ is
finite-dimensional for each finitely generated projective $P$ and each $N'_p$.
So the spaces $\Ext^1(K_{LR}(M), N'_p)$ are all finite-dimensional and
the ${\lim}^1$ term in the above sequence
vanishes by the Mittag-Leffler lemma. It remains
to be shown that each $\Ext^2(K_{LR}(M), N'_p)$ vanishes. For this, 
since $N'_p$ is of finite length, it
suffices to check that $\Ext^2(K_{LR}(M), S_x)$ vanishes for each
vertex $x$ of $\Z Q$. Indeed, by Lemma~\ref{cor:RHom}, we have
\[
\Ext^2(K_{LR}(M), S_x) = D \Hom(S_{\tau^{-1}(x)}, K_{LR}(M))
\]
and this last space vanishes because $K_{LR}(M)$ is a submodule
of $K_R(M)$.
\end{proof}

\begin{theorem} \label{thm:representability}
If $M$ is a finite-dimensional $\cs_C$-module, then $KK(M)$ is a finitely
generated projective $\cR_C$-module and $CK(M)$ a finitely cogenerated
injective $\cR_C$-module.
\end{theorem}

\begin{proof} Since $M$ is a finite-dimensional $\cs_C$-module,
it is finitely generated. Thus the module $K_L(M)$ is also finitely
generated and hence right bounded and pointwise finite-dimensional.
These properties are inherited by the submodule $KK(M)$ of
$K_L(M)$. They imply that there is a surjection $P \to KK(M)$ where
$P$ is a direct sum of projectives $u_i^\wedge$, $i\in I$, such that
the family of the vertices $u_i$ is right bounded and each vertex
$x$ of $\Z \tilde{Q}$ occurs at most finitely many times among the
$u_i$, $i\in I$. It follows that $P$ is right bounded and pointwise
finite-dimensional and so is the kernel $M'$ of $P\to M$. Thus,
by Lemma~\ref{lemma:exactness}, the functor $\Hom(KK(M),?)$ takes
the sequence
\[
0 \to M' \to P \to KK(M) \to 0
\]
to an exact sequence. Thus, the morphism $P \to KK(M)$ splits
and $KK(M)$ is a direct factor of $P$. Let us show that $P$ is finitely
generated. First we notice that $K_{LR}(M)$ is finite-dimensional since
it is pointwise finite-dimensional and both right and left bounded. Therefore,
for each vertex $x$ of $\Z Q$, by Corollary~\ref{cor:RHom}, the space
\[
\Hom(KK(M), S_x) = \Ext^1(K_{LR}(M), S_x)
\]
is finite-dimensional and vanishes for all but finitely many vertices $x$.
Hence $KK(M)$ must be a finite sum of projectives $x^\wedge$.
The proof of the second
assertion is dual.
\end{proof}

\subsection{Description of the strata} \label{ss:description-of-the-strata}
Our goal in this section is to prove the description of the
strata of $\cm_0(w)$ given in Proposition~\ref{prop:description-strata}. 
We use the notations and assumptions of section~\ref{ss:construction-of-Phi}.
We define an arbitrary (not necessarily finite-dimensional) $\cR_C$-module $U$ to be
{\em stable} if we have $\Hom(N,U)=0$ for each module $N$ in the
kernel $\cn$ of the restriction functor $\res$. If $U$ is finite-dimensional,
it is stable iff $\Hom(S_x, U)=0$ for each vertex $x$ of $\Z Q$.
Dually, $U$ is {\em co-stable}
if $\Hom(U,N)=0$ for each $N$ in $\cn$. Clearly submodules of 
stable modules are stable and quotient modules of co-stable modules
are co-stable. Now let us fix an $\cs_C$-module $M$. 
By the adjunctions between $K_L$, $\res$ and $K_R$,
for each $\cs_C$-module $M$, the module $K_R(M)$ is stable and
$K_L(M)$ is co-stable. Since $K_{LR}(M)$ is a submodule of $K_R(M)$
and a quotient of $K_L(M)$, it is both stable and co-stable. This yields
the first statement of Proposition~\ref{prop:description-strata}. If we 
apply the functor $\res$ to the canonical morphism
\[
\can: K_L(M) \to K_R(M)
\]
we obtain the composition of the adjunction morphisms
\[
\res K_L(M) \liso M \iso \res K_R(M)
\]
and in particular the restriction $\res(\can)$ is invertible. Since the
restriction functor is exact, it also makes the canonical morphisms
\[
\xymatrix{K_L(M) \ar@{->>}[r] & K_{LR}(M)\,\, \ar@{>->}[r] & K_R(M) }
\]
invertible and so $\res(K_{LR}(M))$ is canonically isomorphic to $M$.
Now let us assume that $M$ is finite-dimensional of dimension vector $w$
and belongs to $\cm_0(w)$.
Then $K_{LR}(M)$ is finite-dimensional: Indeed, $K_L(M)$ is
right bounded, $K_R(M)$ is left bounded and both are pointwise finite-dimensional.
Let $(v,w)$ be the dimension vector of $K_{LR}(M)$. Since $\res(K_{LR}(M))$
is isomorphic to $M$, the second component of the dimension vector
of $K_{LR}(M)$ is indeed the dimension vector
$w$ of $M$. This also shows that the image under $\pi$ of the point $\tilde{M}$ of 
$\cm^{reg}(v,w)$ corresponding to $K_{LR}(M)$ equals $M$, which therefore 
does belong to the stratum $\pi(\cm^{reg}(v,w))$. Now the other assertions of 
Proposition~\ref{prop:description-strata} are immediate from the facts recalled 
in section~\ref{ss:stratification}.

\subsection{Intermediate extensions and closed orbits} 
\label{ss:intermediate-extensions-closed-orbits} Our goal in this
section is to prove Proposition~\ref{prop:intermediate-extension-closed-orbit}.
We use the notations and assumptions of section~\ref{ss:construction-of-Phi}.
If $M$ is a finite-dimensional $\cR$-module of dimension vector $(v,w)$,
the {\em $G_v$-orbit of $M$} is the orbit corresponding to $M$ in
the affine variety $\rep(\cR^{op}, v,w)$ of representations of $\cR^{op}$
with dimension vector $(v,w)$. By abuse of language, we say that 
a $G_v$-stable subset of $\rep(\cR^{op}, v,w)$ {\em contains} a module,
if it contains the orbit corresponding to the module.

\begin{lemma} \label{lemma:Gv-orbit-closure}
Let 
\[
\xymatrix{ 
0 \ar[r] & L \ar[r] & M \ar[r] & N \ar[r] & 0
}
\]
be an exact sequence of finite-dimensional $\cR_C$-modules.
If we have $\res(L)=0$ (resp. $\res(N)=0$), then the closure of the 
$G_v$-orbit of $M$ contains $L_{ss}\oplus N$ (resp. $L\oplus N_{ss}$), 
where $L_{ss}$ is the semi-simple module with
the same dimension vector as $L$. 
\end{lemma}

\begin{proof} For each vertex $x$ of $\Z Q$, we choose an isomorphism
$M_x = L_x \oplus N_x$ which provides a splitting of the sequence
\[
\xymatrix{ 
0 \ar[r] & L_x \ar[r] & M_x \ar[r] & N_x \ar[r] & 0.
}
\]
For an invertible scalar $t$, let $g(t)$ be the element of $G_v$ which
acts by $t \oplus \id$ on $L_x \oplus N_x$. When $t$ tends to zero,
the representation $g(t).x$, where $x$ is a point in the orbit of $M$, 
tends to $L\oplus N$. Since $L$ is a nilpotent
representation of $k(\Z Q)$, its $G_v$-orbit contains $L_{ss}$ in its
closure. Thus, the $G_v$ orbit of $M$ contains $L_{ss}\oplus N$
in its closure. The proof of the second statement is analogous.
\end{proof}

Let us now prove Proposition~\ref{prop:intermediate-extension-closed-orbit}.
Since $L$ is stable, we have an exact sequence
\[
\xymatrix{
0 \ar[r] & K_{LR}(\res L) \ar[r] & L \ar[r] & N \ar[r] & 0,
}
\]
where $\res(N)=0$. By the Lemma, the closure of the $G_v$-orbit of 
$L$ contains $K_{LR}(\res L)\oplus N_{ss}$. Now let $U$ be a module
in the unique closed $G_v$-orbit in the closure of the $G_v$-orbit of $L$
(the existence of this unique orbit is guaranteed by geometric invariant theory).
Let us show that the orbit of $U$ contains $K_{LR}(\res L) \oplus N_{ss}$.
We have $\res(U)=\res(L)$ since the restriction to $\cs_C$ of a module
is constant on the closure of its $G_v$-orbit. Thus, the morphisms
\[
K_L(\res(M)) \to K_R(\res(M)) \mbox{ and }
K_L(\res(U)) \to K_R(\res(U)) 
\]
are equal and so $K_L(\res(M)) \to K_R(\res(M))$ factors through
the adjunction morphism $\eps : U \to K_R(\res U)$. Therefore, the
module $K_{LR}(\res M)$ is contained in $\im(\eps) \subset K_R(\res U)$.
Let $i$ denote the inclusion $K_{LR}(\res M) \subset \im(\eps)$.
Now by the lemma, the closure of the orbit of $U$ contains
$\im(\eps)\oplus (\ker(\eps))_{ss}$ and the closure of the orbit
of $\im(\eps)$ contains $K_{LR}(\res M)\oplus (\cok(i))_{ss}$.
Thus the orbit of $U$, which equals its closure, contains the object
\[
K_{LR}(\res M)\oplus (\cok(i))_{ss} \oplus (\ker(\eps))_{ss}.
\]
This shows that $U$ is isomorphic to $K_{LR}(\res M)\oplus N_{ss}$,
as claimed.

\subsection{Characterization of the strata} \label{ss:characterization-of-the-strata}
Our goal in this section is to prove the characterization of the strata
of $\cm_0(w)$ given in Theorem~\ref{thm:strata}. 
We use the notations and assumptions of section~\ref{ss:construction-of-Phi}.
We need the following lemmas.
For a vector $v: \Z Q_0 \to \Z$, we define $C_q v : \Z Q_0 \to \Z$ by
\[
(C_q v)(x) = v(x) - \big(\sum_{y\to x} v(y)\big) + v(\tau(x))\ko x\in \Z Q_0 \ko
\]
where the sum ranges over all arrows $y\to x$ of $\Z Q$. The index $q$
reminds us that $C_q$ is a `quantum Cartan matrix', cf. section~3.1 of
\cite{Nakajima11}. Notice that the linear map $v \mapsto C_q v$ is injective on 
the space of finitely-supported vectors. 

\begin{lemma}  \label{lemma:dim-Ext1}
Let $U$ be a finite-dimensional $\cR_C$-module of dimension vector $(v,w)$.
If $U$ is stable and co-stable, the vector
\[
x \mapsto \dim \Ext^1(S_x, U) \ko x\in \Z Q_0 \ko
\]
equals $w\sigma - C_q v$. 
\end{lemma}

\begin{proof} By part~d) of Corollary~\ref{cor:RHom}, the space
$\Ext^1(S_x,U)$ is the homology in degree~$1$ of the complex
\[
0 \to U(x) \to \bigoplus_{y\to x} U(y) \to U(\tau(x)) \to 0 \ko
\]
where the sum ranges over all arrows $y \to x$ of $\cR_C$. 
Since $U$ is stable and co-stable, the homologies in degree $0$
and $2$ of this complex vanish. Thus, the dimension of the
homology in degree $1$ equals
\[
-\dim U(x) + \big(\sum_{y\to x} \dim U(y) \big) - \dim U(\tau(x)) = -(C_q v)(x) + w(\sigma(x)).
\]
\end{proof}

\begin{lemma} \label{lemma:decomp-PhiM}
Let $M$ be an $\cs_C$-module. Let $(v,w)$ be the dimension
vector of $K_{LR}(M)$. Then, for each vertex $x$ of $\Z Q$, the multiplicity
of the indecomposable $H(x)$ in $\Phi(M)$ equals
\[
dim \Ext^1(S_x, K_{LR}(M)) = (w\sigma -C_q v)(x).
\]
\end{lemma}

\begin{proof} Recall from section~\ref{ss:construction-of-Phi} that
$CK(M)$ isomorphic to the finitely cogenerated injective module
\[
D\Hom(\Phi(M), H(?)).
\]
The multiplicity of an indecomposable $H(x)$ in $\Phi(M)$ equals
the multiplicity of the injective indecomposable
\[
D\Hom(H(x), H(?)) = D k(\Z Q)(x, ?)
\]
in $CK(M)$. This multiplicity equals the multiplicity of the simple $S_x$
in the socle of $CK(M)$, that is to say the dimension of $\Hom(S_x, CK(M))$.
Now by Lemma~\ref{lemma:Hom-Ext-KLR}, we have the isomorphism
\[
\Hom(S_x, CK(M)) = \Ext^1(S_x, K_{LR}(M)).
\]
By Lemma~\ref{lemma:dim-Ext1}, the dimension of the right hand
side equals $(w\sigma - C_q v)(x)$.
\end{proof}

Theorem~\ref{thm:strata} is now an easy consequence: 
Let $w$ be a dimension vector for $\cs_C$ (i.e. $w$ vanishes on the vertices
not belonging to $C$). Proposition~\ref{prop:description-strata} shows
that two $\cs_C$-modules $M_1$ and $M_2$ belong to the same
stratum of $\cm_0(w)$ iff the $\cR_C$-modules $K_{LR}(M_1)$ and
$K_{LR}(M_2)$ have the same dimension vector $(v,w)$ and
in this case, the objects $\Phi(M_1)$ and $\Phi(M_2)$ are isomorphic,
by Lemma~\ref{lemma:decomp-PhiM}. Conversely, if $\Phi(M_1)$ and
$\Phi(M_2)$ are isomorphic, then by the same Lemma, we have
\[
w\sigma - C_q v_1 = w\sigma - C_q v_2 \ko
\]
where $(w,v_i)$ is the dimension vector of $K_{LR}(M_i)$, $i=1,2$. Since
$C_q$ is injective on the space of dimension vectors, we find that
$K_{LR}(M_1)$ and $K_{LR}(M_2)$ have the same dimension vector,
which implies that $M_1$ and $M_2$ lie in the same stratum,
by Proposition~\ref{prop:description-strata}.

\subsection{Resolution of the intermediate extension}
\label{ss:resolution-of-the-intermediate-extension}
For future reference, we record the following lemma:
\begin{lemma} \label{lemma:resolution-of-the-intermediate-extension}
Let $M$ be a finite-dimensional $\cs_C$-module.
\begin{itemize}
\item[a)] The $\cR_C$-module $K_{LR}(M)$
has a minimal injective resolution with finitely cogenerated terms
\[
\xymatrix{ 0 \ar[r] & K_{LR}(M) \ar[r] & I^0 \ar[r] & I^1 \ar[r] & 0 \ko}
\]
where $I^0$ is the direct sum of the modules $\sigma(x)^\vee$
with multiplicity equal to $\dim \Hom(S_{\sigma(x)}, M)$, $x\in \Z Q_0$,
and $I^1$ contains the summand $x^\vee$ with multiplicity
equal to the multiplicity of $H(x)$ as a direct factor of $\Phi(M)$, $x\in\Z Q_0$.
\item[b)] The $\cR_C$-module $K_{LR}(M)$
has a minimal projective resolution with finitely generated terms
\[
\xymatrix{ 0 \ar[r] & P_1 \ar[r] & P_0 \ar[r] & K_{LR}(M) \ar[r] & 0 \ko}
\]
where $P_0$ is the direct sum of the modules $\sigma(x)^\wedge$
with multiplicity equal to 
$\dim \Hom(M, S_{\sigma(x)})$, $x\in \Z Q_0$,
and $P_1$ contains the summand $x^\wedge$ with
multiplicity equal to the multiplicity of $H(\tau^{-1}(x))$ as a
direct factor of $\Phi(x)$.
\end{itemize}
\end{lemma}

\begin{remark} One can show that if the projective $\cs_C$-modules
coincide with the injective ones, then $I^1$ contains no direct factor
$\sigma(x)^\vee$ and $P_1$ no direct factor $\sigma(x)^\wedge$.
\end{remark}

\begin{proof} a) The module $K_{LR}(M)$ is
finite-dimensional (cf.~section~\ref{ss:description-of-the-strata}) and
thus admits an injective resolution with finitely cogenerated terms $I^p$
and the multiplicity of the indecomposable injective $u^\vee$ associated
with a vertex $u$ of $\Z \tilde{Q}$ equals the dimension of
$\Ext^p_{\cR_C}(S_u, K_{LR}(M))$. Let $x$ be a vertex of $\Z Q$. We have
$\Hom(S_x,K_{LR}(M))=0$ since $K_{LR}(M)$ is stable. 
Since we have $\Hom(S_{\sigma(x)}, CK(M))=0$, we find
\[
\Hom(S_{\sigma(x)}, K_{LR}(M)) = \Hom(S_{\sigma(x)}, K_R(M)) = 
\Hom_{\cs_C}(S_{\sigma(x)}, M).
\]
The multiplicity of $x^\vee$ in $I^1$ equals $\Ext^1_{\cR_C}(S_x, K_{LR}(M))$
and this equals the multiplicity of $H(x)$ in $\Phi(M)$ by 
Lemma~\ref{lemma:decomp-PhiM}. The proof of b) is similar but
uses the duality isomorphism~(\ref{eq:duality}) in addition.
\end{proof}

\subsection{On the degeneration order} \label{ss:degeneration-order}
Our goal in this section is to prove Theorem~\ref{thm:degeneration-order}.
We may and will assume that $Q$ is connected.
Let $(v,w)$ and
$(v',w)$ be dimension vectors of $\cR_C$ associated with
non empty subsets $\cm^{reg}(v,w)$ and $\cm^{reg}(v',w)$
of the corresponding smooth quiver varieties.
Let $M$ and $M'$ be $\cs_C$-modules belonging to the corresponding
strata. Let us assume that $\Phi(M) \leq \Phi(M')$ in the
degeneration order of \cite{JensenSuZimmermann05a}
and show that the stratum $\pi(\cm^{reg}(v',w))$ is contained in the 
closure of $\pi(\cm^{reg}(v,w))$. Recall from Corollary~4.1.3.14
of \cite{Qin12} that this is the case iff we have $v'(x)\leq v(x)$
for all vertices $x$ of $\Z Q_0$. Now by Proposition~\ref{prop:description-strata},
if we denote by $\dimv U$ the dimension vector of a module $U$, 
we have
\[
(v,w)=\dimv K_{LR}(M) \mbox{ and } (v',w) = \dimv K_{LR}(M').
\]
So we need to show the inequality
\[
\dimv K_{LR}(M') \leq \dimv K_{LR}(M).
\]
Indeed, by definition \cite{JensenSuZimmermann05a}, the relation $\Phi(M)\leq \Phi(M')$
means that there is an object $Z$ of $\cd_Q$ and a triangle
\begin{equation} \label{eq:degen-triangle}
\Phi(M') \to \Phi(M) \oplus Z \to Z \to \Sigma \Phi(M').
\end{equation}
If $Q$ is a Dynkin quiver, then 
thanks to the triangle equivalence of Theorem~\ref{thm:equivalence-with-DQ}, we
can find a finite-dimensional $\cs_C$-module $U$ such that
$\Phi(U)$ is isomorphic to $Z$. If $Q$ is not a Dynkin quiver, then
a priori, the object $Z$ may not belong to the image of $\Phi$ but
we claim that we can always replace it with an object in the image.
For this, let $H^i_\ca$ denote the homology functors for the heart $\ca$ of $\cd_Q$
whose indecomposable objects are the indecomposable regular $kQ$-modules
and all the objects $\tau^p P$, where $p\in\Z$ and $P$ is an
indecomposable projective $kQ$-module. Then $\Phi(M)$ and 
$\Phi(M')$ lie in $\ca$ but $Z$ may have non vanishing homologies
in several degrees. However, if we apply $H^{-1}_\ca$ to the 
triangle~(\ref{eq:degen-triangle}), we find the exact sequence
\[
0 \to H^{-1}_\ca(Z) \to H^{-1}_\ca(Z) \to \Phi(M') \to \Phi(M)\oplus H^0_\ca(Z).
\]
Since $H^{-1}_\ca(Z)$ is finite-dimensional, the second morphism of
the sequence must be invertible and so the map $\Phi(M') \to \Phi(M)\oplus H^0_\ca(Z)$
is injective. Thus, the following sequence is left exact
\begin{equation} \label{eq:degen-ex-seq}
0 \to \Phi(M') \to \Phi(M) \oplus H^0_\ca(Z) \to H^0_\ca(Z) \to 0.
\end{equation}
It is also right exact because $H^1_\ca(\Phi(M'))$ vanishes. The category
$\ca$ contains the subcategory of all regular $kQ$-modules as a
torsion subcategory and the corresponding category of torsion-free
objects is the category $\cv$ 
formed by the direct sums of the objects $\tau^p(P)$, where
$p\in\Z$ and $P$ is an indecomposable projective $kQ$-module. Since
$\Phi(M')$ is torsion-free, the map 
\[
\Phi(M) \oplus H^0_\ca(Z) \to H^0_\ca(Z)
\]
induces an isomorphism in the torsion parts. Thus, if we apply
the functor $A \to A_{tf}$ (which takes an object to its torsion-free quotient) to the
exact sequence~(\ref{eq:degen-ex-seq}), we obtain an exact sequence
\[
0 \to \Phi(M') \to \Phi(M) \oplus H^0_\ca(Z)_{tf} \to H^0_\ca(Z)_{tf} \to 0.
\]
So after replacing $Z$ with $H^0_\ca(Z)_{tf}$, we have a short exact
sequence of objects in $\cv$:
\begin{equation} \label{eq:degen-ses}
0 \to \Phi(M') \to \Phi(M) \oplus Z \to Z \to 0.
\end{equation}
By the first part of Theorem~\ref{thm:strata}, each object of $\cv$
is isomorphic to the image under $\Phi$ of a finite-dimensional
semi-simple $\cs_C$-module. So we can find a finite-dimensional semi-simple
$\cs_C$-module $U$ such that $\Phi(U)$ is isomorphic to $Z$.

From now on, $Q$ may be Dynkin or non Dynkin. By the surjectivity
at the level of extensions stated in Lemma~\ref{lemma:surjectivity-Ext1},
respectively in Corollary~\ref{cor:surjectivity-ext1-non-Dynkin},
we can lift the triangle~(\ref{eq:degen-triangle}),
respectively the short exact sequence~(\ref{eq:degen-ses}),
to a short exact sequence of $\cs_C$-modules
\[
0 \to M' \to E \to U \to 0.
\]
Since $K_L$ is right exact and $K_R$ is left exact, the image
\[
0 \to K_{LR}(M') \to K_{LR}(E) \to K_{LR}(U) \to 0
\]
of this sequence is exact at the terms $K_{LR}(M')$ and $K_{LR}(U)$ but
cannot be expected to be exact at $K_{LR}(E)$. Thus, we find the inequality
\[
\dimv(K_{LR}(M')) + \dimv(K_{LR}(U)) \leq \dimv(K_{LR}(E)).
\]
Now we also know that we have
\[
\dimv E = \dimv M' + \dimv U = \dimv M + \dimv U = \dimv(M\oplus U)
\]
and that $\Phi(E)$ is isomorphic to $\Phi(M) \oplus Z = \Phi(M) \oplus \Phi(U) =
\Phi(M\oplus U)$. By Lemma~\ref{lemma:decomp-PhiM} and the 
injectivity of the map $v \mapsto C_q v$, we conclude that we have
\[
\dimv(K_{LR}(E)) = \dimv(K_{LR}(M\oplus U)).
\]
Thus, we obtain the inequality
\[
\dimv K_{LR}(M) \geq \dimv K_{LR}(M')
\]
as claimed.

Conversely, suppose that we have $v'(x)\leq v(x)$ for all vertices $x\in \Z Q_0$.
Then the class of $\Phi(M')$ in the split Grothendieck group $K_0^{split}(\cd_Q)$
is obtained from that of $\Phi(M)$ by adding a positive integer linear combination
of elements of the form
\begin{equation} \label{eq:AR-triangle-relation}
[U] - [E] + [\tau^{-1}(U)]
\end{equation}
associated with Auslander--Reiten triangles
\begin{equation}\label{eq:AR-triangle}
\xymatrix{ 
U \ar[r] & E \ar[r] & \tau^{-1}(U) \ar[r] & \Sigma U
}
\end{equation}
for indecomposables $U$ of $\cd_Q$. By the transitivity of the degeneration
relation, we may assume that the class of $\Phi(M')$ is obtained from that
of $\Phi(M)$ by adding a single element~(\ref{eq:AR-triangle-relation}). This means
that we have decompositions
\[
\Phi(M) \iso V\oplus E \mbox{ and } \Phi(M') \iso V \oplus U \oplus \tau^{-1}(U)
\]
for some $V$ in $\cd_Q$. Now if we add a split triangle over the identity of
$V\oplus \tau^{-1}(U)$ to the triangle~(\ref{eq:AR-triangle}), we obtain
a triangle
\[
\xymatrix@C=0.5cm{
V\oplus \tau^{-1}(U)\oplus U \ar[r] & V\oplus \tau^{-1}(U)\oplus E \ar[r] & 
\tau^{-1}(U) \ar[r] & \Sigma (V\oplus \tau^{-1} (U)\oplus U) \ko
}
\]
and this is of the form
\[
\xymatrix{
\Phi(M') \ar[r] & \Phi(M) \oplus Z \ar[r] & Z \ar[r] & \Phi(M').
}
\]
Thus, we have $\Phi(M) \leq \Phi(M')$ as claimed.

\subsection{Description of the fibers} \label{ss:description-of-the-fibres}
Our goal in this section is to prove Theorem~\ref{thm:description-of-the-fibres}.
We first determine which fibres are non empty.
Let $w$ be a dimension vector of $\cs_C$ (i.e. a dimension vector of $\cs$
whose support is contained in $C$). Let $M$ be a point of $\cm_0(w)$.
Let $L_0=K_{LR}(M)$ and $(v_0,w)=\dimv L_0$. Recall from
section~\ref{ss:construction-of-Phi} that
\[
CK(M) = K_R(M)/K_{LR}(M)
\]
is an injective module over $k(\Z Q)$ isomorphic to $D\cd_Q(\Phi(M),?)$. 
For a dimension vector $u$ of $k(\Z Q)$, let $\Gr_u(CK(M))$ denote the
quiver Grassmannian of $k(\Z Q)$-submodules $N \subset CK(M)$ such
that $\dimv N = u$.

\begin{lemma} \label{lemma:non-empty-fibres}
Let $v$ be a dimension vector of $k(\Z Q)$. The fiber of $\pi: \cm(v,w) \to \cm_0(w)$
over $M$ is non empty iff the quiver Grassmannian $\Gr_{v-v_0}(CK(M))$ 
is non empty.
\end{lemma}

\begin{proof} Suppose the fiber is non empty. Then there is a stable $\cR_C$-module
$L$ with $\dimv L = (v,w)$ and $\res(L)=M$. The adjunction morphisms yield a
commutative diagram
\[
\xymatrix@R=0.5cm{
K_{L}(M) \ar@{=}[dd] \ar@{->>}[rd] \ar[rr]^{\can} & & K_R(M) \ar@{=}[dd] \\
  & *++{K_{LR}(M)} \ar@{>->}[ru] \ar@{>..>}[d]& \\
K_L(\res L) \ar[r]_-{\eps L} & *++{L} \ar@{>->}[r]_-{\eta L} & K_R (\res L)
}
\]
Here the map $\eta L$ is injective since $L$ is stable. Since the canonical
morphism $K_L(L) \to K_R(L)$ equals $(\eta L)(\eps L)$, its image is
contained in that of $\eta L$ and we obtain a canonical injection
$K_{LR}(M) \to L$. The quotient $L/K_{LR}(M)$ is isomorphic to
the image $\ol{L}$ of $L$ in $CK(M)=K_R(M)/K_{LR}(M)$, which
is a submodule of dimension $v-v_0$. Conversely, if $U\subset CK(M)$
is a submodule of dimension $v-v_0$, its inverse image $L\subset K_R(M)$
is a stable $\cR_C$-module of dimension vector $(v,w)$ such that
$\res(L) = M$.
\end{proof}

\begin{lemma} \label{lemma:fibre-is-quiver-grassmannian}
Let $v$ be a dimension vector of $k(\Z Q)$. The fibre of 
$\pi: \cm(v,w) \to \cm_0(w)$ over $M$ is homeomorphic, in the complex-analytic
topology, to the quiver Grassmannian 
\[
\Gr_{v-v_0}(CK(M)) \mbox{ of } CK(M) = D \cd_Q(\Phi(M),?).
\]
\end{lemma}

\begin{proof} By the previous lemma, we can assume that the fiber is non empty.
Hence $w  \sigma - C_q v_0$ has non negative components: these components
indicate the multiplicities of the indecomposable factors of $\Phi(M)$ by
Lemma~\ref{lemma:decomp-PhiM}. Consider the following dimension 
vector of $\cs_C$:
\[
w_0= w - (C_q v_0) \sigma^{-1}.
\]
Let $S$ be the semi-simple $\cs_C$-module of dimension vector $w_0$.
By Nakajima's slice Theorem (Theorem~2.4.9 of \cite{KimuraQin12} based on
Theorem~3.14 of \cite{Nakajima11} based on \S 3.3 of \cite{Nakajima01}), 
the fibre of $\pi: \cm(v,w) \to \cm_0(w)$ over $M$ is isomorphic, in the complex-analytic 
topology, to the fibre of
\[
\pi: \cm(v-v_0, w_0) \to \cm_0(w_0)
\]
over $S$. Moreover, it follows from Lemma~\ref{lemma:decomp-PhiM}, that
$CK(M)$ is isomorphic to $CK(S)$. So it remains to prove the assertion
for $M=S$. In this case, it was shown by Savage--Tingley in
Theorem~5.4 of \cite{SavageTingley11}, who used input from
Shipman's \cite{Shipman10} to improve on a bijection constructed in
the non graded case by Lusztig in Theorem~2.26 of \cite{Lusztig98a}.
\end{proof}

\subsection{Exactness of $\Phi$ in the non Dynkin case}
\label{ss:stratifying-functor-non-Dynkin-case}
Let $Q$ be a connected non Dynkin quiver. Our goal in this
section is to show that the stratifying functor $\Phi$ constructed
in section~\ref{ss:construction-of-Phi} is exact in a suitable sense. 
Let $H: k(\Z Q) \to \cd_Q$ be
Happel's embedding (Theorem~\ref{thm:Happel}). Its image consists of
the $\tau$-orbits in $\cd_Q$ of the indecomposable projective $kQ$-modules.
Recall that $\cv$ denotes the category of all finite direct sums of objects in the image.
The category $\cv$ is the category of `vector bundles' on the `non commutative
curve' whose category of coherent sheaves is the heart of the
$t$-structure on $\cd_Q$ whose left aisle consists of the
objects $X$ such that $H^1(X)$ is a preinjective $kQ$-module
and $H^p(X)$ vanishes for all $p\geq 2$. In particular, $\cv$ is an
exact category, whose conflations are the sequences which give
rise to triangles in $\cd_Q$. 

\begin{theorem} \label{thm:exact-functor-Phi}
The functor $\Phi: \mod(\cs_C) \to \cv$ is exact.
\end{theorem}

\begin{lemma} \label{lemma:KL-KR-exact}
The functors $K_L$ and $K_R: \Mod(\cs_C) \to \Mod(\cR_C)$
are exact.
\end{lemma}

\begin{proof}[Proof of the Lemma] By applying the restriction functor
to the sequences of part b) of Theorem~\ref{thm:res-simple-SC-modules}
we see that $\res(x^\wedge)=\res(P_C(x))$ is projective and 
$\res(x^\vee)=\res(I_C(x))$ is injective
for each non frozen vertex $x$. Moreover, the restriction of 
$\sigma(x)^\wedge$ is clearly projective and that of $\sigma(x)^\vee$
injective. It follows that the restriction functor preserves projectivity and
thus its right adjoint $K_R$ is exact. Moreover, we see that $\res$ takes
finitely cogenerated injective modules to injective modules. Now in order
to check whether a sequence is exact, it suffices to check whether
its image under $\Hom(?,I)$ is exact for each finitely cogenerated
injective module $I$. Thus, the left adjoint $K_L$ of $\res$ is
also exact.
\end{proof}

\begin{proof}[Proof of the Theorem] Let
\begin{equation} \label{eq:ex-seq-modSc}
0 \to M' \to M \to M'' \to 0
\end{equation}
be an exact sequence of $\mod(\cs_C)$. We know 
(for example from the appendix of \cite{Keller90}), that in order to show 
that the sequence
\[
0 \to \Phi(M') \to \Phi(M) \to \Phi(M'') \to 0
\]
is a conflation of the exact category $\cv$, it suffices to show that the sequence
\[
0 \to \Hom(\Phi(M''),?) \to \Hom(\Phi(M),?) \to \Hom(\Phi(M'),?) \to 0
\]
is exact in the abelian category of left exact functors 
$\Lex(\cv^{op})\subset \Mod(\cv^{op})$.
By Lemma~\ref{lemma:KL-KR-exact}, the images of the
sequence (\ref{eq:ex-seq-modSc}) under $K_L$ and $K_R$ are exact. By the snake lemma, we thus
have an exact sequence
\begin{align*}
0 &\to KK(M') \to KK(M) \to KK(M'')  \\
   &\to CK(M') \to CK(M) \to CK(M'') \to 0.
\end{align*}
Since we have
\[
CK(M)(x) = D\Hom(\Phi(M), H(x)) \ko
\]
we deduce that the sequence
\[
0 \to \Hom(\Phi(M''),?) \to \Hom(\Phi(M),?) \to \Hom(\Phi(M'),?)
\]
is exact in the category $\Mod(\cv^{op})$ and thus in $\Lex(\cv^{op})$. 
In order to show that the morphism
\[
\Hom(\Phi(M),?) \to \Hom(\Phi(M'		),?)
\]
is epimorphic in $\Lex(\cv^{op})$ it suffices to show that its cokernel $U$
in $\Mod(\cv^{op})$ is effaceable, i.e. that for each element
$f$ of $U(V)$, $V\in \cv$, there is an inflation $V \to V'$ such that
the map $U(V) \to U(V')$ takes $f$ to zero. Now the cokernel $U$ is a subfunctor 
of the $k$-dual of  $KK(M'')$, and $KK(M'')$ is right bounded. Now
all the Auslander-Reiten sequences
\[
0 \to H(x) \to \bigoplus_{x\to y} H(y) \to H(\tau^{-1} (x)) \to 0
\]
where $x$ is a vertex of $\Z Q_0$ and the sum ranges over all arrows
of $\Z Q$ with source $x$, are conflations of $\cv$. In particular, the maps
\[
0 \to H(x) \to \bigoplus_{x\to y} H(y) 
\]
are inflations. This shows that each right bounded left $\cv$-module
is effaceable. Thus, the $k$-dual of $KK(M'')$ and its submodule $U$
are effaceable, as claimed.
\end{proof}

\begin{corollary} \label{cor:surjectivity-ext1-non-Dynkin}
Let $U$ be a finite-dimensional semi-simple $\cs_C$-module
and $M$ a finite-dimensional $\cs_C$-module. Then $\Phi$
induces a surjection
\[
\Ext^1_{\cs_C}(U,M) \to \Ext^1_{\cd_Q}(\Phi(U), \Phi(M)).
\]
\end{corollary}

\begin{proof} We may assume that $U$ is a simple $\cs_C$-module
$S_{\sigma^{-1}(x)}$. If $M$ is also a simple $\cs_C$-module
$S_{\sigma^{-1}(y)}$, the claim is easy to check using part b)
of Theorem~\ref{thm:res-simple-SC-modules}. For the general
case, we use induction on the length of $M$ and the fact
that both functors $\Ext^1_{\cs_C}(U,?)$ and $\Ext^1_{\cd_Q}(\Phi(U), \Phi(?))$
are right exact. For $\Ext^1_{\cs_C}(U,?)$, this is again a
consequence of part b)
of Theorem~\ref{thm:res-simple-SC-modules}, which yields
a resolution of length $1$.
\end{proof}

\section{The Dynkin case}
\label{s:stratifying-functor-Dynkin-case}

\subsection{The singularity category of $\cs_C$} \label{ss:sing-of-SC}
Suppose that $Q$ is a connected acyclic quiver. By part b) of
Theorem~\ref{thm:res-simple-SC-modules}, if $Q$ is not a Dynkin
quiver, the category $\cs_C$ is the path category of an (infinite) quiver
and thus is of global dimension one. On the other hand, if $Q$ is
a Dynkin quiver, by part~a) of Theorem~\ref{thm:res-simple-SC-modules}, the
category $\cs_C$ is of infinite global dimension. We will show that
its singularity category (defined via Gorenstein projective/injective modules)
is equivalent to the derived category of $Q$. We will then use this
equivalence to construct the stratifying functor $\Phi: \mod\cs_C \to \cd_Q$
as outlined in section~\ref{ss:Gorenstein-homological-algebra}.

\subsection{Construction of resolutions} \label{ss:construction-of-resolutions}
From now on, we suppose that $Q$ is a Dynkin quiver and $C$ a configuration
in $\Z Q$ satisfying Assumption~\ref{main-assumption}. We have the
restriction functor and its right and left adjoints, which we denote by the
same symbols as in the case where $C=\Z Q_0$ considered in
section~\ref{ss:descr-Phi-Kan-extensions}.
\[
\xymatrix{
\Mod(\cR_C) \ar[d]|-*+{{\scriptstyle \res}} \\
\Mod(\cs_C) \ar@<2ex>[u]^{K_L} \ar@<-2ex>[u]_{K_R}
}
\]
The Kan extensions $K_L$ and $K_R$ are fully faithful so that
$\res$ is a localization of abelian categories in the sense of
\cite{Gabriel62}. Recall from Lemma~\ref{lemma:characterization-image-KR-KL}
that an object $M$ belongs to the
image of $K_R$ if and only if, for each $\cR_C$-module
$N$ with $\res(N)=0$, we have $\Hom(N,M)=0$ and $\Ext^1(N,M)=0$.

\begin{lemma} \label{lemma:crit-closed}
\begin{itemize} 
\item[a)] Each $\cR_C$-module $N$ with $\res(N)=0$ is the
union of its submodules of finite length.
\item[b)] An $\cR_C$-module $M$ belongs to the image of $K_R$ if
and only if we have 
\[
\Hom(S_x, M)=0 \mbox{ and }\Ext^1(S_x, M)=0
\]
for each non frozen vertex $x$.
\end{itemize}
\end{lemma}

\begin{proof} a) Let $N$ be a $\cR_C$-module such that $\res(N)=0$. 
Then $N$ is a module over the quotient of $\cR_C$ by the ideal
generated by the identities of the objects $\sigma(x)$, $x\in\Z Q_0$.
Now this quotient is equivalent to the category of indecomposable
objects of the derived category $\cd_Q$. Since $Q$ is a Dynkin
quiver, the projective $\cd_Q$-modules $\cd_Q(?,u)$, $u\in \cd_Q$,
are of finite length. Thus, each $\cd_Q$-module is the union of
its submodules of finite length and the same holds for the $\cR_C$-modules
whose restriction to $\cs_C$ vanishes. Thus, the claim holds for $N$.

b) Of course, the condition is necessary. Suppose conversely that
it holds for some $\cR_C$-module $M$. By part a), it follows that
we have $\Hom(N,M)=0$ for each $\cR_C$-module $N$ with
$\res(N)=0$. Thus, the adjunction morphism $M \to K_R \res(M)$ is
injective and we have an exact sequence
\[
0 \to M \to K_R \res(M) \to M' \to 0 \ko
\]
where $\res(M')=0$. We have the exact sequence
\[
\Hom(S_x, K_R \res(M)) \to \Hom(S_x, M') \to \Ext^1(S_x, M) \ko
\]
where the first and the last term vanish. Thus, the module $M'$ has
no submodules of finite length. By part a), we must have $M'=0$.
\end{proof}

\begin{lemma} \label{lemma:KLP-iso-KRP}
\begin{itemize}
\item[a)] For each  finitely generated projective
$\cs_C$-module $P$,  the canonical morphism $K_L P \to K_R P$ is
invertible.
\item[b)] For each finitely generated projective $\cR_C$-module $P$, the
canonical morphism $P \to K_R(\res P)$ is invertible. 
\item[c)] For each finitely generated projective $\cR_C$-module $P$, the
module 
\[
K_{LR}(\res(P))
\]
is isomorphic to the submodule of $P$ generated by
the images of all morphisms $\sigma(x)^\wedge \to P$, $x\in C$.
\end{itemize}
\end{lemma}

\begin{proof} a) It suffices to show that $K_L P$ belongs to the image of
$K_R$. We check the condition of part~b) of Lemma~\ref{lemma:crit-closed}.
Let $z$ be a non frozen vertex. By part e) of Corollary~\ref{cor:RHom}, we have
\[
D\RHom(S_z, K_L P) = \RHom(K_L P, \Sigma^2 S_{\tau^{-1}(z)}).
\]
This last object vanishes since $K_L P$ is a direct sum of projectives
$\sigma(y)^\wedge$, $y\in\Z Q_0$, and $\tau^{-1}(z)$ is a non frozen
vertex.

b) By part a), it suffices to prove the assertion for $P=\tau(x)^\wedge$ for
any vertex $x$ of $\Z Q$. If we apply $K_R \circ \res$ to the exact sequence
\[
0 \to \tau(x)^\wedge \to \sigma(x)^\wedge \to  S_{\sigma(x)} \to 0 \ko
\]
we obtain the exact sequence
\[
0 \to K_R(\res(\tau(x)^\wedge)) \to K_R(\sigma(x)^\wedge) \to K_R(S_{\sigma(x)}).
\]
By part a), we have an isomorphism 
$\sigma(x)^\wedge \iso K_R(\sigma(x)^\wedge)$ and we have a monomorphism
$K_R(S_{\sigma(x)}) \to \sigma(x)^\vee$. Thus, we have an exact sequence
\[
0 \to K_R(\res(\tau(x)^\wedge)) \to \sigma(x)^\wedge \to \sigma(x)^\vee
\]
and one checks that the morphism $\sigma(x)^\wedge \to \sigma(x)^\vee$ 
is non zero. Thus, its image is $S_{\sigma(x)}$ and we find that 
$K_R(\res(\tau(x)^\wedge))$ is the kernel of $\sigma(x)^\wedge \to S_{\sigma(x)}$.
But this is $\tau(x)^\wedge$.

c) By definition, $K_{LR}(P)$ is the sum of the images in $K_R(P)$ of all the
morphisms 
\[
K_L(\res(\sigma(x)^\wedge)) \to K_R(\res(P))
\] 
induced by morphisms $\sigma(x)^\wedge \to P$. Now trivially, we have 
$K_L(\res(\sigma(x)^\wedge)) \iso \sigma(x)^\wedge$ and by part a),
we have $P \iso K_R(\res(P))$. This implies the claim.
\end{proof}

The following lemma will be of great use.

\begin{lemma} \label{lemma:acyclic-three-term}
Let $x$ be a vertex of $\Z Q$. Let $P$ be a finitely
generated projective $\cs_C$-module and $I$ a finitely
cogenerated injective $\cs_C$-module. 
\begin{itemize}
\item[a)] The image under $\Hom(\res(?), P)$ of the resolution
\begin{equation} 
0 \to (\Sigma^{-1} x)^\wedge \to P_C(x) \to x^\wedge \to 0
\end{equation}
of $x^\wedge_\cd$ constructed in Theorem~\ref{thm:rc-resolutions} 
is acyclic.
\item[b)] The image under $\Hom(I,\res(?))$ of the coresolution
\begin{equation} 
0 \to x^\vee \to I_C(x) \to (\Sigma x)^\vee \to 0.
\end{equation}
of $x^\vee_\cd$ constructed in Theorem~\ref{thm:rc-resolutions} 
is acyclic.
\end{itemize}
\end{lemma}

\begin{proof} a) We have $\Hom(\res(?), P) =\Hom(?, K_R(P))$ and
by Lemma~\ref{lemma:KLP-iso-KRP}, we know that $K_R(P)$ is isomorphic
to $K_L(P)$, which is a finite direct sum of projective $\cR_C$-modules
$\sigma(y)^\wedge$ associated with the vertices $y$ of $\Z Q$. So the
claim is that $\RHom(x^\wedge_\cd, \sigma(y)^\wedge)$ vanishes.
Since $x^\wedge_\cd$ is a finite-dimensional module concentrated
on non frozen vertices, it suffices to show that $\RHom(S_z, \sigma(y)^\wedge)$
vanishes for each non frozen vertex $z$. Now by part  e)
of Corollary~\ref{cor:RHom}, we have
\[
D\RHom(S_z, \sigma(y)^\wedge) = \RHom(\sigma(y)^\wedge, \Sigma^2 S_{\tau(z)})
\]
and the last object is isomorphic to a shift of $D S_{\tau(z)}(\sigma(y))=0$. The
proof of b) is dual.
\end{proof}

\subsection{The weak Gorenstein property} 
\label{ss:Gorenstein-property}
The following lemma implies that the category $\cs_C$ is {\em weakly Gorenstein
of dimension $1$} in the sense that we have
\[
\Ext^p_{\cs_C}(M,P)=0=\Ext^p_{\cs_C}(I,M)
\]
for all $p\geq 2$ and 
each finite-dimensional module $M$, each finitely generated projective
module $P$ and each finitely cogenerated injective module $I$. 

\begin{lemma} \label{lemma:Gorenstein-dimension-1} 
\begin{itemize}
\item[a)] We have $\Ext^p_{\cs_C}(I,M)=0$ for all $p\geq 2$, for
each finitely cogenerated injective $\cs_C$-module $I$ and each pointwise finite-dimensional right bounded $\cs_C$-module $M$. 
\item[b)] We have $\Ext^p_{\cs_C}(M,P)=0$ for all $p\geq 2$, for
each finitely generated projective $\cs_C$-module $P$ and each pointwise finite-dimensional left bounded $\cs_C$-module $M$. 
\end{itemize}
\end{lemma}

\begin{proof} a) We may and will assume that $I=\sigma(x)^\vee$ for some vertex
$x$ in $C$. 

{\em First step: If $M$ is finite-dimensional, then $\Ext^p_{\cs_C}(\sigma(x)^\vee, M)$ is finite-dimensional for all integers $p$ and vanishes for all $p\geq 2$.}
It suffices to prove the statement for a simple module $M=S_{\sigma(y)}$.
Now the injective resolution of $S_{\sigma(y)}$ in part a) of
Theorem~\ref{thm:res-simple-SC-modules} is the complex
of $\cs_C$-modules
\[
0 \to \sigma(y)^\vee \to I_C(y) \to I_C(\Sigma y) \to I_C(\Sigma^2 y) \to \ldots
\]
which is spliced together from $\sigma(y)^\vee \to y^\vee$ and the
sequences
\[
(\Sigma^{p-1} y)^\vee \to I_C(\Sigma^{p-1} y) \to (\Sigma^p y)^\vee \ko p\geq 1 \ko
\]
which are extracted from the co-resolutions
\[
0 \to (\Sigma^{p-1} y)^\vee_\cd \to (\Sigma^{p-1} y)^\vee \to I_C(\Sigma^{p-1} y)
\to (\Sigma^p y)^\vee \to 0
\]
constructed in part a) of Theorem~\ref{thm:rc-resolutions}. Now
the fact that $\Ext^p_{\cs_C}(\sigma(x)^\vee, S_\sigma(y))$ vanishes
for $p\geq 2$ follows from Lemma~\ref{lemma:acyclic-three-term}.

{\em Second step: The claim.}
Since $M$ is right-bounded and pointwise finite-dimensional, it is
the inverse limit of a countable system
\[
\ldots \to M_{i} \to M_{i-1} \to \ldots \to M_1 \to M_0
\]
of finite-dimensional modules. We have 
\[
\RHom(\sigma(x)^\vee, M) = \Rlim \RHom(\sigma(x)^\vee, M_i).
\]
Since the homology of each complex $\RHom(\sigma(x)^\vee, M_i)$ is
finite-dimensional (by the first step), we obtain that
\[
\Ext^p(\sigma(x)^\vee, M) = \lim \Ext^p(\sigma(x)^\vee, M_i)
\]
for each integer $p$. By the first step, this implies the claim.

The proof of b) is dual.
\end{proof}

\begin{question} Are the injective $\cs_C$-modules of projective
dimension at most~$1$ and the projective $\cs_C$-modules of
injective dimension at most~$1$? \end{question}

We do not know the answer if $C$ is the set of all vertices of $\Z Q$.
On the other hand, in certain cases, the classes of projective and injective
$\cs_C$-modules coincide, for example when $C$ is chosen so that
$\cs_C$ is the category of projective modules over the repetitive
algebra of an algebra $B$ derived equivalent to the Dynkin
quiver $Q$ (in particular if $B=kQ$ as in Leclerc-Plamondon's
\cite{LeclercPlamondon12}), cf.~also section~\ref{ss:Frobenius-models}.

\subsection{Coherence} \label{ss:coherence} We consider
the category $\cs_C=\cs$ associated with the set $C$ of all
vertices of $\Z Q$. Let $\ct$ be the full subcategory of $\cR$
whose objects are all the vertices of $\Z Q$.
The following Proposition implies in particular
part a) of Proposition~\ref{prop:properties-singular-category}.

\begin{proposition} 
\begin{itemize}
\item[a)] The category $\ct$ is hereditary and thus coherent.
\item[b)] The category $\cR$ is coherent.
\item[c)] The category $\cs$ is coherent.
\end{itemize}
\end{proposition}

\begin{remark}
We do not know under what conditions on $C$ the category
$\cs_C$ is coherent. Clearly, this holds if $\cs_C$ happens to
be locally bounded, i.e. if, for each object $x$ of $\cs_C$, there
are at most finitely many objects $y$ such that $\cs_C(x,y)\neq 0$
or $\cs_C(y,x)\neq 0$. By the proposition, it also holds for $C= \Z Q_0$.
\end{remark}

\begin{proof}[Proof of the proposition] a) Let $\tilde{\cR}$ be 
the path category of $\Z \tilde{Q}$ and
$\tilde{\ct}$ the path category of $\Z Q$. The projection $\tilde{\cR} \to \cR$
induces a functor $P: \tilde{\ct} \to \ct$. It is not hard to see that
there is a well-defined inverse functor $S: \ct \to \tilde{\ct}$ such that
\[
S(\alpha_x \beta_x) = - \sum_{i=1}^s \alpha_i \sigma(\alpha_i)
\]
whenever we have a mesh of $\Z \tilde{Q}$ with arrows
\[
\xymatrix{ 
\tau(x) \ar[r]^{\beta_x} & \sigma(x) \ar[r]^{\alpha_x} & x
}
\]
and arrows
\[
\xymatrix{\tau(x) \ar[r]^{\sigma(\alpha_i)} & y_i \ar[r]^{\alpha_i} & x} \ko 
1\leq i\leq s.
\]
Thus, the category $\ct$ is isomorphic to $\tilde{\ct}$, which is
hereditary since it is the path category of a quiver.

b) Let $f: P_1 \to P_0$ be a morphism in $\proj(\cR)$.
We need to show that its kernel is finitely generated. Since it is
a submodule of $P_1$, it is pointwise finite-dimensional and
right bounded. Thus, it has a projective cover and it suffices
to show that $\Hom(\ker(f), S_u)$ vanishes for all but finitely
many vertices $u$ of $\Z \tilde{Q}$. We first consider vertices
$u$ of the form $\sigma(x)$ for some vertex $x$ of $\Z Q$.
We have the exact sequence
\[
\xymatrix{
0 \ar[r] & \ker(f) \ar[r] & P_1 \ar[r]^f & \im(f) \ar[r] & 0
}
\]
and deduce the exactness of the sequence
\[
\Hom(P_1, S_u) \to \Hom(\ker(f), S_u) \to \Ext^1(\im(f), S_u) \to 0.
\]
Thus, it suffices to show that $\Ext^1(\im(f), S_u)$ vanishes for
all but finitely many $u$. We have
\[
0 \to \im(f) \to P_0 \to \cok(f) \to 0
\]
and so we have
\[
\Ext^1(\im(f), S_u) \iso \Ext^2(\cok(f), S_u).
\]
Now for $u=\sigma(x)$, the module $S_u$ is of injective
dimension at most one and so both terms vanish. We deduce
that $\Hom(P_1, S_u) \to \Hom(\ker(f), S_u)$ is surjective. Thus,
there are at most finitely many vertices $u=\sigma(x)$ such
that $\Hom(\ker(f), S_u)$ is non zero. It remains to study
the case where $u=x$ for some vertex $x$ of $\Z Q$. Now
since $S_x$ is a $\ct$-module, we have an injection
\[
\Hom_{\cR}(\ker(f), S_x) \subset \Hom_{\ct}(\res_\ct(\ker(f)), S_x).
\]
So it suffices to show that the right hand term vanishes
for almost all vertices $x$ of $\Z Q$. Now $\res_\ct(\ker(f))$ identifies
with the kernel of the restriction $\res_\ct(f) : \res_\ct(P_1) \to \res_\ct(P_0)$.
The restriction of a module $\sigma(x)^\wedge$ to $\ct$ is isomorphic
to $\tau(x)^\wedge$ and the restriction of a module $x^\wedge_\cR$ to
$x^\wedge_\ct$. Thus the restrictions of $P_0$ and $P_1$ to $\ct$
are finitely generated projective and by part a), the kernel
$\ker(\res_\ct(f))$ is finitely generated. This shows that
$\Hom(\res_\ct(\ker(f)), S_x)$ vanishes for almost all vertices
$x$ of $\Z Q$. 

c) Let $f: P_0 \to P_1$ be a morphism of $\proj(\cs)$. Then 
\[
f \ten_\cs \cR : P_1 \ten_\cs \cR \to P_0 \ten_\cs \cR
\]
is a morphism of $\proj(\cR)$ and its restriction to $\cs$ identifies
with $f$. We have
\[
\ker(f) \iso \res_\cs(\ker(f \ten_\cs \cR)).
\]
By part b), the module $\ker(f\ten_\cs \cR)$ is finitely generated. Now
the claim follows because for each vertex $u$ of $\Z \tilde{Q}$, the
module $\res_\cs(u^\wedge_\cR)$ is finitely generated: This is
clear for the vertices $u=\sigma(x)$, $x\in \Z Q_0$; for the
vertices $u=x$, $x\in \Z Q$, it follows from part a) of
Theorem~\ref{thm:rc-resolutions}.
\end{proof}

\subsection{Two Frobenius categories} \label{ss:two-Frobenius-categories}
Recall that, for a $k$-category $\cc$, 
a $\cc$-module $M$ is {\em Gorenstein projective} \cite{EnochsJenda95} 
if there is an acyclic complex
\[
P: \ldots \to P_1 \to P_0 \to P_{-1} \to \ldots
\]
of finitely generated projective modules such that $M$ is isomorphic
to the cokernel of $P_1 \to P_0$ and that the complex $\Hom(P,P')$ is
still acyclic for each finitely generated projective $\cc$-module $P'$.
Dually, a $\cc$-module $M$ is {\em Gorenstein injective} if there
is an ayclic complex of finitely cogenerated injective $\cc$-modules
\[
I: \ldots \to I^{-1} \to I^0 \to I^1 \to \ldots
\]
such that $M$ is isomorphic to the kernel of $I^0 \to I^1$ and the complex
$\Hom(I', I)$ is still acyclic for each finitely cogenerated 
injective $\cs_C$-module $I'$. By Proposition~5.1 of 
\cite{AuslanderReiten91}, the full subcategories $\gpr(\cc)$ and
$\gin(\cc)$ formed by the Gorenstein projective, respectively injective,
modules are closed under extensions in $\Mod(\cc)$. It then
follows easily that they are Frobenius exact categories and
that their subcategories of projective--injective objects are
the subcategory of finitely generated projective $\cc$-modules,
respectively finitely cogenerated injective $\cc$-modules.

For each $\cs_C$-module $M$, choose exact sequences
\[
0 \to \Omega M \to P_M \to M \to 0 
\quad\mbox{and}\quad
0 \to M \to I^M \to \Sigma M \to 0
\]
where $P_M$ is projective and $I^M$ injective. For example,
if $x$ is a vertex of $\Z Q$, we can use the restrictions to 
$\cs_C$ of the sequences of $\cR_C$-modules
\[
0 \to x^\wedge \to \sigma^{-1}(x)^\wedge \to S_{\sigma^{-1}(x)} \to 0
\quad \mbox{and} \quad
0 \to S_{\sigma(x)} \to \sigma(x)^\vee \to x^\vee \to 0
\]
so that $\Omega S_{\sigma^{-1}(x)} = \res(x^\wedge)$ and
$\Sigma S_{\sigma(x)} = \res(x^\vee)$.

\begin{lemma} If $M$ is a finite-dimensional $\cs_C$-module, then
the module $\Omega M$ is Gorenstein projective and the module
$\Sigma M$ is Gorenstein injective.
\end{lemma}

\begin{proof} Since the category $\gpr(\cs_C)$ is closed under
extensions in $\Mod(\cs_C)$ it suffices to prove the claim
when $M$ is a simple module associated with a vertex in $C$. 
Let $P$ be the complex obtained by splicing together the
restrictions to $\cs_C$ of the sequences
\[
0 \to (\Sigma^{p-1} x)^\wedge \to P(\Sigma^p x) \to (\Sigma^p x)^\wedge \to 0
\]
extracted from the resolution of $(\Sigma^p x)^\wedge_\cd$ from 
Theorem~\ref{thm:rc-resolutions}, where $p\in\Z$. By 
Lemma~\ref{lemma:acyclic-three-term}, the complex $\Hom(P,P')$ is
still acyclic for each finitely generated $\cs_C$-module $P'$. 
Hence $\Omega S_{\sigma^{-1}(x)} = \res(x^\wedge)$ is
Gorenstein projective. The proof for $\Sigma S_{\sigma(x)}$ is dual.
\end{proof}

For future reference, we record the following easy consequences
of the definition of Gorenstein modules.
\begin{lemma} \label{lemma:computation-ext}
Let $x$ be a non frozen vertex. Let $L$ be a Gorenstein
projective module and $M$ a Gorenstein injective module. We have
isomorphisms
\begin{align*}
\Ext^1(L, \Omega S_{\sigma^{-1}(x)}) &= \Hom(L, S_{\sigma^{-1}(x)})/\Hom(L, \sigma^{-1}(x)^\wedge) \ko \\
\Ext^1(\Sigma S_{\sigma(x)}, M) &= \Hom(S_{\sigma(x)}, M)/\Hom(\sigma(x)^\vee, M).
\end{align*}
\end{lemma}

\begin{proof} This follows at once by applying the functor $\Hom(L,?)$ to the
sequence
\[
0 \to \Omega S_{\sigma^{-1}(x)} \to \sigma^{-1}(x)^\wedge \to S_{\sigma^{-1}(x)} \to 0
\]
and using the fact that $\Ext^1(L, \sigma^{-1}(x)^\wedge)$ vanishes 
because $L$ is Gorenstein projective. Similarly for the second isomorphism.
\end{proof}

\begin{lemma} \label{lemma:surjectivity-Ext1}
Let $L$ and $M$ be finite-dimensional $\cs_C$-modules. Let $L \to I$ be an
injection into an injective module and $P \to M$ a surjection from a projective module.
We have canonical isomorphisms
\begin{align*}
\Ext^1(\Sigma L, \Sigma M)  & \iso \Ext^1(L,M)/\Ext^1(I,M) \mbox{ and } \\
\Ext^1(\Omega L, \Omega M)  & \iso \Ext^1(L,M)/\Ext^1(L,P).
\end{align*}
\end{lemma}

\begin{proof} We have 
$
\Ext^1(\Sigma L, \Sigma M) \iso \Ext^2(\Sigma L, M)
$.
If we apply $\Ext^1(?,M)$ to the sequence 
\[
0 \to L \to I \to \Sigma L \to 0 \ko
\]
we get the first isomorphism. The proof of the second one is dual.
\end{proof}

\subsection{Link to the derived category} \label{ss:link-to-the-derived-category}
The stable categories $\ul{\gpr}(\cs_C)$ and $\ul{\gin}(\cs_C)$
of the Frobenius categories $\gpr(\cs_C)$ respectively $\gin(\cs_C)$
are canonically triangulated. In accordance with our overall convention,
we write $\Sigma$ for their suspension functors. We will show that
these categories are triangle equivalent to $\cd_Q$.

\begin{lemma} \label{lemma:comparison-of-ext}
For all vertices $x$, $y$ of $\Z Q$ and all integers $p$, we have 
isomorphisms
\begin{align*}
\ul{\gin}(\cs_C)(\Sigma S_{\sigma(x)}, \Sigma^p \Sigma S_{\sigma(y)})
 & = \cd_Q(H(x), \Sigma^p H(y)) \\
 & = \ul{\gpr}(\cs_C)(\Omega S_{\sigma^{-1}(x)}, \Sigma^p \Omega S_{\sigma^{-1}(y)}).
\end{align*}
\end{lemma}

\begin{proof} {\em First step: The claim for $p\geq 2$.}
Let us abbreviate $\Sigma S_{\sigma(x)}$ by $Fx$.
We have
\[
\ul{\gin}(\cs_C)(Fx, \Sigma^p Fy)
=\Ext^p_{\cs_C}(Fx,\Sigma S_{\sigma(y)})
=\Ext^{p+1}_{\cs_C}(\Sigma S_{\sigma(x)}, S_{\sigma(y)}).
\]
Now by the sequence
\[
0 \to M \to I^M \to \Sigma M \to 0 \ko
\]
where $M=S_{\sigma(x)}$, and part~a) of Lemma~\ref{lemma:Gorenstein-dimension-1},
this last space is isomorphic to
\[
\Ext^{p}_{\cs_C}(S_{\sigma(x)}, S_{\sigma(y)})=\cd_Q(H(x), \Sigma^p H(y)) \ko
\]
where we have used Corollary~\ref{cor:ext-between-simples-in-sc} 
for the last isomorphism.

{\em Second step: The claim for arbitrary $p$.} Let us first note that by
the sequences
\[
0 \to (\Sigma^{p-1} y)^\vee \to I_C(\Sigma^p y) \to (\Sigma^p y)^\vee \to 0 \ko
\]
of Theorem~\ref{thm:rc-resolutions}, which become exact when restricted to $\cs_C$, 
we have isomorphisms in $\ul{\gin}(\cs_C)$:
\[
\Sigma^m Fy=\Sigma^m\Sigma S_{\sigma(y)}=\Sigma^m \res(y^\vee) = \res((\Sigma^m y)^\vee)=
\Sigma S_{\sigma(\Sigma^m y)} = F \Sigma^m y
\]
for all $m\in \Z$. We deduce that, for a given $p\in\Z$ and any $q\geq 2-p$, we have
\[
\ul{\gin}(\cs_C)(Fx, \Sigma^p Fy)=
\ul{\gin}(\cs_C)(Fx, 
\Sigma^{p+q} F(\Sigma^{-q}y)).
\]
By the first step, this last space is isomorphic to
\[
\cd_Q(H(x), \Sigma^{p+q} (H(\Sigma^{-q}y))) = \cd_Q(H(x), \Sigma^{p} H(y)).
\]
The proof of the second isomorphism is analogous.
\end{proof}

\begin{theorem}  \label{thm:equivalence-with-DQ}
There are  triangle equivalences
\[
F: \cd_Q \iso \ul{\gin}(\cs_C) \quad \mbox{respectively} \quad
F': \cd_Q \iso \ul{\gpr}(\cs_C)
\]
taking $H(x)$ to $\Sigma S_{\sigma(x)}$ respectively
$\Omega S_{\sigma(x)}$ (sic!) for each vertex $x$ of $\Z Q$.
\end{theorem}

\begin{remark} Let $kQ$ denote the path category of $Q$ considered
as a full subcategory of $\cR_C$ via the embedding $i \mapsto (0,i)$. We
have a functor $kQ \to \gpr(\cs_C)$ taking $x$ to $\res(x^\wedge)$.
It gives rise to a $kQ$-$\cs_C$-bimodule $X$ given by
\[
X(u,x) = \Hom(u^\wedge, \res(x^\wedge)) , \quad x\in Q_0 , u \in \sigma(C).
\]
Since $\gpr(\cs_C)$ is a Frobenius category, we have a canonical
triangle functor
\[
\can: \cd^b(\gpr(\cs_C)) \to \ul{\gpr}(\cs_C)
\]
cf.~for example \cite{KellerVossieck87}, \cite{Rickard89b}. Now
we can describe the equivalence $F': \cd_Q \iso \ul{\gpr}(\cs_C)$ as the
composition
\[
\xymatrix@C=1.2cm{
\cd_Q \ar[r]^-{?\lten_{kQ} X} &  \cd^b(\gpr(\cs_C)) \ar[r]^-{\can} &  \ul{\gpr}(\cs_C).
}
\]
Of course, there is an analogous description for $F$.
\end{remark}

\begin{proof} By Lemma~\ref{lemma:comparison-of-ext}, when $x$ and
$y$ are vertices of $Q$ and $p$ is a non zero integer, we have
\[
\ul{\gin}(\cs_C)(\Sigma S_{\sigma(x)}, \Sigma^p S_{\sigma(y)})=
\cd_Q(H(x), \Sigma^p H(y)) =0.
\]
Moreover, the endomorphism algebra of the sum of the $\Sigma S_{\sigma(x)}$,
$x\in Q_0$, is isomorphic to the path algebra $kQ$. 
Since $\ul{\gin}(\cs_C)$ is an algebraic triangulated category, it follows
that we have a fully faithful triangle functor $F: \cd_Q \to \ul{\gin}(\cs_C)$
taking $H(x)$ to $\Sigma S_{\sigma(x)}$ for each $x\in Q_0$. 
Now recall that for an arbitrary vertex $x$ of $\Z Q$, the module 
$\Sigma S_{\sigma(x)}$ is isomorphic to $\res(x^\vee)$. By restricting
the co-resolution
\[
0 \to S_x \to x^\vee \to \bigoplus_{x\to y} y^\vee \to \tau^{-1}(x)^\vee \to 0
\]
of part d) of Lemma~\ref{lemma:resolutions} to $\cs_C$ we obtain an
exact sequence
\[
0 \to  \res(x^\vee) \to \bigoplus_{x\to y} \res(y^\vee) \to \res(\tau^{-1}(x)^\vee) \to 0.
\]
Starting from the vertices of the slice $Q \subset \Z Q$ and `knitting' to the
left and to the right, we use the triangles associated with these sequences
to check that $F$ takes each vertex $x$ of $\Z Q$ to $\res(x^\vee) =\Sigma S_{\sigma(x)}$.

Now fix an object $M$ in $\ul{\gin}(\cs_C)$. Clearly the functor
\[
\Hom(F(?), M) : \cd_Q^{op} \to \Mod k
\]
is cohomological. Moreover, the isomorphism
\[
\Hom(F H(x), M) = \Hom(S_{\sigma(x)}, M)/\Hom(\sigma(x)^\vee, M) \ko
\]
where morphisms on the left are taken in $\ul{\gin}(\cs_C)$ and those
on the right in $\Mod(\cs_C)$, shows that $\Hom(F H(x), M)$ is only
non zero if $S_{\sigma(x)}$ occurs in the socle of $M$ and that its
dimension is bounded by the dimension of the socle of $M$. Since
$M$ is a submodule of a finite sum of modules $\sigma(y)^\vee$, $y\in \Z Q_0$,
it follows that $\Hom(F(?), M)$ takes values in the finite-dimensional
vector spaces and vanishes on all but finitely many indecomposable
objects of $\cd_Q$. In particular, $\Hom(F(?), M)$ is a finitely generated
cohomological functor on $\cd_Q$. This implies that it is representable. Thus,
the functor $F$ admits a right adjoint $F_\rho$ and for each object
$M$ of $\ul{\gin}(\cs_C)$, we have a canonical triangle
\[
F F_\rho M \to M \to GM \to \Sigma F F_\rho M \ko
\]
where $GM$ is right orthogonal to the image of $\cd_Q$ under $F$. We
will show that this right orthogonal subcategory vanishes. Indeed, suppose
that $M$ is an object in the right orthogonal. Since $M$ is a submodule
of a finite direct sum of modules $\sigma(y)^\vee$, $y\in \Z Q_0$, it
has a finite-dimensional socle. We proceed by induction on its dimension.
If it is zero, then $M$ has to be zero. So suppose that $S_{\sigma(x)}$ is
a simple submodule in the socle of $M$. Since $M$ is right orthogonal
to the image of $F$, we have
\[
0=\Hom(F H(x), M) = \Hom(S_{\sigma(x)}, M)/\Hom(\sigma(x)^\vee, M).
\]
Thus the inclusion $S_{\sigma(x)} \to M$ extends to a map
$\sigma(x)^\vee\to M$. This map is injective since it induces an injection
in the socles. Since $\sigma(x)^\vee$ is an injective module, it is actually
a direct summand and $M$ is the direct sum of $\sigma(x)^\vee$ and
a submodule $M'$, whose socle is of strictly smaller dimension than
that of $M$ and which still belongs to the right orthogonal of the image of $F$.
By the induction hypothesis, $M'$ must be injective and so $M$ is injective. 
\end{proof}

\subsection{Description of $\Phi$ via Kan extensions} 
\label{ss:description-of-Phi-via-Kan-extensions}
Let 
\[
\Phi: \mod(\cs_C) \to \cd_Q
\]
be the composition of $\Omega: \mod(\cs_C) \to \ul{\gpr}(\cs_C)$
with the quasi-inverse of the equivalence $F' : \cd_Q \to \ul{\gpr}(\cs_C)$
of Theorem~\ref{thm:equivalence-with-DQ}. Notice that $\Phi$
is a $\delta$-functor as the composition of the $\delta$-functor
$\Omega$ with a triangle equivalence. Equivalently, we could
define $\Phi$ as the composition of $\Sigma: \mod(\cs_C) \to \ul{\gin}(\cs_C)$
with the quasi-inverse of the equivalence $F: \cd_Q \to \ul{\gin}(\cs_C)$.

Let us now prove Proposition~\ref{prop:description-KK-CK}, which claims
that for $M\in\mod(\cs)$, we have functorial isomorphisms 
of $k(\Z Q)$-modules
\[
KK(M) = \Hom_{\cd_Q}(H(?), \tau\Phi(M)) 
\mbox{ and } 
CK(M) = D \Hom_{\cd_Q}(\Phi(M), H(?))  \ko
\]
where $H$ is Happel's embedding (Theorem~\ref{thm:Happel}).

\begin{proof}[Proof of Proposition~\ref{prop:description-KK-CK}] 
We only prove the second isomorphism, the proof of
the first one being similar.
Let $P\to M$ be a surjection with finitely generated projective $P$.

{\em First step: We have a canonical isomorphism
\[
\cok(K_R(P) \to K_R(M)) \iso CK(M) \ko
\]
} 
Indeed, by definition,
$CK(M)$ is the cokernel of the canonical morphism $K_L(M) \to K_R(M)$.
Now we have a commutative diagram
\[
\xymatrix{
K_L(P) \ar[d] \ar[r] & K_R(P) \ar[d] \\
K_L(M) \ar[r] & K_R(M).
}
\]
Here the top horizontal arrow is an isomorphism by Lemma~\ref{lemma:KLP-iso-KRP}
and the left vertical arrow is surjective since $K_L$ is right exact. The claim follows.

{\em Second step: For each vertex $x$ of $\Z Q$, we have a canonical 
isomorphism
\[
(K_R(M)/K_R(P))(x) \iso \Ext^1(S_{\sigma^{-1}(x)}, M)/\Ext^1(S_{\sigma^{-1}(x)}, P).
\]
}
Recall the sequence
\[
0 \to \res(x^\wedge) \to \sigma^{-1}(x)^\wedge \to S_{\sigma^{-1}(x)} \to 0.
\]
It shows that $\Omega(S_{\sigma^{-1}(x)})$ is isomorphic to $\res(x^\wedge)$.
Now we have isomorphisms
\[
(K_R(M))(x) = \Hom(x^\wedge, K_R(M)) = \Hom(\res(x^\wedge), M) =
\Hom(\Omega S_{\sigma^{-1}(x)}, M)
\]
and similarly for $P$ instead of $M$. Now by definition, we have
\[
\Ext^1(S_{\sigma^{-1}(x)}, M) = 
\Hom(\Omega S_{\sigma^{-1}(x)}, M)/\Hom(\sigma^{-1}(x)^\wedge, M)
\]
and similarly for $P$ instead of $M$. The claim follows because each
morphism $\sigma^{-1}(x)^\wedge \to M$ factors through $P\to M$. 

{\em Third step: For each vertex $x$ of $\Z Q$, we have canonical 
isomorphisms
\begin{align*}
\Ext^1(S_{\sigma^{-1}(x)}, M)/\Ext^1(S_{\sigma^{-1}(x)}, P)  & \iso
\ul{\gpr}(\cs_C)(\Omega S_{\sigma^{-1}(x)}, \Sigma \Omega M) \\
& = D \cd_Q(\Phi M, H(x)).
\end{align*}
}
Indeed, since $\Ext^2(S_{\sigma^{-1}(x)}, P)$ vanishes
(Lemma~\ref{lemma:Gorenstein-dimension-1}), we have an isomorphism
\[
\Ext^1(S_{\sigma^{-1}(x)}, M)/\Ext^1(S_{\sigma^{-1}(x)}, P) \iso
\Ext^2(S_{\sigma^{-1}(x)}, \Omega M).
\]
Now we clearly have an isomorphism
\[
\Ext^2(S_{\sigma^{-1}(x)}, \Omega M) = \Ext^1(\Omega S_{\sigma^{-1}(x)}, \Omega M)
\]
and the last space identifies with 
\[
\ul{\gpr}(\cs_C)(\Omega S_{\sigma^{-1}(x)}, \Sigma \Omega M).
\]
Now we have $F' \Phi M = \Omega M$ and 
\[
F'(\tau^{-1} H(x)) = F'(H(\tau^{-1}(x)))= \Omega S_{\sigma(\tau^{-1}(x))} = 
\Omega S_{\sigma^{-1}(x)} \ko
\]
whence the isomorphism
\[
\ul{\gpr}(\cs_C)(\Omega S_{\sigma^{-1}(x)}, \Sigma \Omega M)=
\cd_Q(\tau^{-1} H(x), \Sigma\Phi M).
\]
Finally, we get
\begin{align*}
\cd_Q(\tau^{-1} H(x), \Sigma \Phi M) &= \cd_Q(H(x), \tau \Sigma \Phi M) 
= \cd_Q(H(x), \nu \Phi M) \\
&= D\cd_Q(\Phi M, H(x)).
\end{align*}
\end{proof}

\subsection{The regular category as a Gorenstein--Auslander category}
\label{ss:regular-category-as-Auslander-category} 
Our goal in this section is to prove 
Theorem~\ref{thm:regular-category-as-Auslander-category}.
We need the following lemma.

\begin{lemma} \label{lemma:kernel-cokernel-adjunction} Let $x$ be a
vertex of $\Z Q$. The adjunction morphisms fit into exact sequences
\[
0 \to x^\vee_\cd \to x^\vee \to K_R \res(x^\vee) \to (\Sigma x)^\vee_\cd \to 0
\]
and
\[
0 \to (\Sigma^{-1} x)^\wedge_\cd \to K_L \res(x^\wedge) \to x^\wedge 
\to x^\wedge_\cd  \to 0.
\]
\end{lemma}

\begin{proof} To compute $K_R \res(x^\vee)$, we use the injective coresolution
\[
\xymatrix{
0 \ar[r] & \res(x^\vee) \ar[r] & I_C(x) \ar[r] & I_C(\Sigma x)
}
\]
obtained by splicing the exact sequences 
\begin{align*}
0 & \to \res(x^\vee) \to I_C(x) \to \res((\Sigma x)^\vee) \to 0 \mbox{ and } \\
0 & \to \res((\Sigma x)^\vee) \to I_C(\Sigma x) \to \res((\Sigma^2 x)^\vee) \to 0
\end{align*}
from part a) of Theorem~\ref{thm:rc-resolutions}. We find that $K_R \res(x^\vee)$
is the kernel of the composition
\[
\xymatrix{
I_C(x) \ar@{->>}[r] & (\Sigma x)^\vee \ar@{->>}[r] &
(\Sigma x)^\vee/(\Sigma x)^\vee_\cd\,\, \ar@{>->}[r] &
I_C(\Sigma x).
}
\]
Thus, $K_R \res(x^\vee)$ is also the kernel of the composition $f$ of the
first two morphisms in this sequence. Now we have the diagram with 
exact rows and columns
\[
\xymatrix{
x^\vee/x^\vee_\cd\,\,\ar@{>->}[r] \ar@{=}[d] & *++{\ker(f)} \ar@{>->}[d] \ar@{->>}[r] & 
*++{(\Sigma x)^\vee_\cd} \ar@{>->}[d] \\
x^\vee/x^\vee_\cd\,\, \ar@{>->}[r]  & I_C(x) \ar@{->>}[d]_f \ar@{->>}[r] & (\Sigma x)^\vee \ar@{->>}[d] \\
   & (\Sigma x)^\vee/(\Sigma x)^\vee_\cd \ar@{=}[r] & (\Sigma x)^\vee/(\Sigma x)^\vee_\cd 
}
\]
This shows the claim. The proof of the second assertion is dual.
\end{proof}

The following theorem implies Theorem~\ref{thm:regular-category-as-Auslander-category}
when we take $C$ to be the set of all vertices of $\Z Q$.

\begin{theorem} \label{thm:regular-category-as-Auslander-category-for-config}
The restriction functor induces
equivalences 
\[
\proj(\cR_C) \to \gpr(\cs_C) \mbox{ and } \inj(\cR_C) \to \gin(\cs_C).
\]
It yields isomorphisms from the quiver $\Z \tilde{Q}_C$ onto
the Auslander--Reiten quivers of $\gpr(\cs_C)$ and $\gin(\cs_C)$
so that the vertices of $C$ correspond to the projective-injective
vertices.
\end{theorem}

\begin{proof}
By Theorem~\ref{thm:equivalence-with-DQ}, each non injective 
indecomposable object of $\gin(\cs_C)$
is of the form $\res(x^\vee)$ for some non frozen vertex $x$ of $\Z \tilde{Q}$ and
of course the indecomposable injective object $\sigma(x)^\vee_{\cs_C}$ is the restriction of
$\sigma(x)^\vee_{\cR_C}$. Thus, the restriction functor is essentially 
surjective. Let us show that it is fully faithful. Let $u$ and $v$ be any
vertices of $\Z \tilde{Q}$. We need to show that the 
adjunction morphism
\[
v^\vee \to K_R(\res v^\vee)
\]
induces a bijection
\[
\Hom(u^\vee, v^\vee) \to \Hom(u^\vee, K_R(\res v^\vee))
\]
If $v=\sigma(y)$ for some non frozen vertex $y$, then the
adjunction morphism $\sigma(v)^\vee \to K_R(\res v^\vee)$ is
itself invertible. So let us assume that $v$ is a non frozen
vertex $y$. By Lemma~\ref{lemma:kernel-cokernel-adjunction},
the adjunction morphism $y^\vee \to K_R(\res y^\vee)$ is the
composition of the epimorphism $p$ in the sequence
\begin{equation} \label{eq:left-half-adjunction}
0 \to y^\vee_\cd \to y^\vee \arr{p} y^\vee/y^\vee_\cd \to 0
\end{equation}
with the monomorphism $i$ in the sequence
\begin{equation} \label{eq:right-half-adjunction}
0 \to y^\vee/y^\vee_\cd \arr{i} K_R(\res y^\vee) \to (\Sigma y)^\vee_\cd \to 0.
\end{equation}
The sequence~(\ref{eq:left-half-adjunction}) yields the exact sequence
\[
\Hom(u^\vee, y^\vee_\cd) \to \Hom(u^\vee, y^\vee) \to \Hom(u^\vee, y^\vee/y^\vee_\cd) \to
\Ext^1(u^\vee, y^\vee_\cd).
\]
Now $y^\vee_\cd$ is an iterated extension of objects $S_z$, $z\in \Z Q_0$.
We have
\[
\RHom(u^\vee, S_z) = D\RHom(\Sigma^{-2} S_{\tau^{-1}(z)}, u^\vee)
\]
by Corollary~\ref{cor:RHom} and so
\[
\Hom(u^\vee, S_z) = D \Ext^2(S_z, u^\vee) =0 \mbox{ and }
\Ext^1(u^\vee, S_z) = D \Ext^1(S_z, u^\vee) =0.
\]
Thus, the map $\Hom(u^\vee, p)$ is bijective. The sequence~(\ref{eq:right-half-adjunction})
yields the exact sequence
\[
0 \to  \Hom(u^\vee, y^\vee/y^\vee_\cd) \to \Hom(u^\vee, K_R(\res(y^\vee)) \to
\Hom(u^\vee,(\Sigma y)^\vee_\cd).
\]
We have $\Hom(u^\vee, (\Sigma y)^\vee_\cd)=0$ because $(\Sigma y)^\vee_\cd$
is also an extension of simples $S_z$, $z\in \Z Q_0$. Thus, the map
$\Hom(u^\vee, i)$ is also bijective and the functor
$\res$ is indeed fully faithful on the subcategory $\inj(\cR_C)$.
The proof for $\proj(\cR_C)$ is dual. The last assertion is clear.
\end{proof}

\subsection{Frobenius models for derived categories of Dynkin quivers}
\label{ss:Frobenius-models}
Let $k$ be an algebraically closed field.
In this section, by a {\em Frobenius category}, we mean a $k$-linear,
$\Hom$-finite Krull--Schmidt category $\ce$
endowed with the structure of an exact category for which it is Frobenius.

Let $Q$ be a Dynkin quiver. A {\em Frobenius model for $\cd_Q$} is
a Frobenius category $\ce$ together with a triangle equivalence
$F: \cd_Q \iso \ul{\ce}$. For example, if $C \subset \Z Q_0$ is a set of vertices
satisfying condition (R) of section~\ref{ss:res-simple-RC-modules}, 
then the category $\ce_C = \gpr(\cs_C)$ becomes a Frobenius
model of $\cd_Q$: It is a Frobenius category by 
section~\ref{ss:two-Frobenius-categories} and 
its stable category is equivalent to $\cd_Q$ by
Theorem~\ref{thm:equivalence-with-DQ}.
Now for an arbitrary Frobenius category $\ce$, consider the following properties:
\begin{itemize}
\item[(P1)] For each indecomposable non projective object $X$ of $\ce$, there is an
almost split sequence starting and an almost split sequence ending at $X$.
\item[(P2)] For each indecomposable projective object $P$ of $\ce$, 
the $\ce$-module $\rad_\ce(?,P)$ and the $\ce^{op}$-module
$\rad_\ce(P,?)$ are finitely generated with simple tops.
\item[(P3)] $\ce$ is standard, i.e. its category of indecomposables is equivalent to
the mesh category of its Auslander--Reiten quiver
(cf.~section~2.3, page~63 of \cite{Ringel84}).
\end{itemize}
The existence of almost split triangles in the stable category $\ul{\ce}$
implies condition (P1) so that this condition holds in particular in all
Frobenius models of $\cd_Q$. 
For $\ce=\ce_C=\gpr(\cs_C)$ as above, the category of indecomposables of
$\ce$ is equivalent to the mesh category $\cR_C$,
by Theorem~\ref{thm:regular-category-as-Auslander-category-for-config}.
We deduce that such categories also satisfy (P2) and (P3).
We do not know Frobenius models of $\cd_Q$ which do not
satisfy (P2). On the other hand, in many cases, condition (P3)
does not hold. For example, let us assume that $Q$ is the quiver 
$1 \to 2 \to 3$ and $A$ the algebra given by the quiver
\[
\xymatrix{1 \ar[r]^\beta & 2 \ar[r]^\alpha & 3}
\]
with the relation $\alpha \beta=0$. Then the category $\cc^b(\proj A)$ of
bounded complexes of finitely generated projective $A$-modules
becomes a Frobenius model for $\cd_Q$ because $A$ is
derived equivalent to the path algebra $kQ$. 
It is not hard to compute the Auslander--Reiten quiver
of the category $\ce=\cc^b(\proj A)$ and to check that
it satisfies (P2). However, it is not standard because the
simple $\ce$-module $S_P$ associated with the complex
\[
P = ( \xymatrix{ \ldots \ar[r] & 0 \ar[r] & P_3 \ar@{=}[r] & P_3 \ar[r] & 0 \ar[r] & \ldots} )
\]
is of projective dimension $2$ whereas in a standard Frobenius
category satisfying (P2), the simple module associated with a projective object
is always of projective dimension $\leq 1$.

The Frobenius models of $\cd_Q$ naturally form a $2$-category:
If $(\ce, F)$ and $(\ce', F')$ are two Frobenius models of $\cd_Q$, a
{\em $1$-morphism} $(\ce, F) \to (\ce', F')$ is an exact functor $G:\ce\to \ce'$
together with an isomorphism $\alpha: \ul{G} \circ F \iso F'$. 
We leave it to the reader to define the $2$-morphisms and to
show that a $1$-morphism is an equivalence in this $2$-category
iff its underlying exact functor is an equivalence. For example,
an inclusion $C \supset C'$ of sets of vertices satisfying $(R)$
yields a $1$-morphism $G: \ce_C \to \ce_{C'}$
which annihilates all indecomposable projectives associated with
the vertices in $C$ but not $C'$. The following corollary results from
section~\ref{ss:two-Frobenius-categories}, 
Theorem~\ref{thm:equivalence-with-DQ}, 
Theorem~\ref{thm:regular-category-as-Auslander-category-for-config}
and the above discussion. 

\begin{corollary} \label{cor:Frobenius-models}
The map taking $C$ to $(\ce_C, F_C)$ induces a bijection from the
set of subsets $C\subset \Z Q_0$ satisfying condition (R)  of 
section~\ref{ss:res-simple-RC-modules} onto the set of
equivalence classes of Frobenius models $(\ce, F)$ of $\cd_Q$
satisfying (P2) and (P3). The inverse bijection sends a Frobenius
model $(\ce, F)$ to the set $C\subset \Z Q_0$ such that the indecomposable
projectives of $\ce$ correspond to the vertices $\sigma^{-1}(c)$,
$c\in C$, of the Auslander--Reiten quiver of $\ce$.
\end{corollary}

\begin{proof} The only thing left to check is that if $\ce$ is a Frobenius
model of $\cd_Q$ satisfying (P2) and (P3), then the corresponding
set $C$ satisfies condition (R). Indeed, let $x$ be a vertex of
$\Z Q$. Let $X$ be the corresponding indecomposable object
of $\ce$. Since $\ce$ is Frobenius, we can find an inflation
$X \to I$, where $I$ is injective. In particular, there is a non zero
morphism from $X$ to an indecomposable injective object.
Thus, there is a path $p$ from $x$ to $\sigma^{-1}(c)$ for some
$c$ in $C$ such that the class of $p$ in $\cR_C$ is non zero. Let us assume,
as we may, that this path is of minimal length. It is the composition
of the canonical arrow $c \to \sigma^{-1}(c)$ with a path $p'$ from
$x$ to $c$. Suppose that the class of $p'$ in $k(\Z Q)$ vanishes.
Then the morphism corresponding to $p'$ vanishes in the 
stable category $\ul{\ce}$. This implies that the class of $p'$ 
in $\cR_C$ is a linear
combination of paths factoring through vertices $\sigma^{-1}(c')$
which lie between $x$ and $c$ for the ordering given by the
existence of a path. But then we obtain a path with non
zero class in $\cR_C$ from $x$ to some $\sigma^{-1}(c')$,
which contradicts the minimality of the length of $p$.
\end{proof}


\def\cprime{$'$} \def\cprime{$'$}
\providecommand{\bysame}{\leavevmode\hbox to3em{\hrulefill}\thinspace}
\providecommand{\MR}{\relax\ifhmode\unskip\space\fi MR }
\providecommand{\MRhref}[2]{%
  \href{http://www.ams.org/mathscinet-getitem?mr=#1}{#2}
}
\providecommand{\href}[2]{#2}

\end{document}